\documentclass{amsart}
\usepackage[centertags]{amsmath}
\usepackage{fancyhdr}
\usepackage{graphicx}
\usepackage{amsfonts,amscd,amsmath,amssymb,amsthm}
\usepackage{times}
\usepackage[centertags]{amsmath}
\DeclareSymbolFont{SY}{U}{psy}{m}{n}
\DeclareMathSymbol{\emptyset}{\mathord}{SY}{'306}

\DeclareMathOperator*{\slim}{s-lim} 
 
 \DeclareMathOperator{\tr}{tr}

\DeclareMathOperator{\les}{\lambda_{\rm ess}}

\DeclareMathOperator{\spec}{spec}

\DeclareMathSymbol{\newtimes}{\mathbin}{SY}{'264}

\title[Relative convergence estimates for the Large Coupling Limit]{
Relative convergence estimates for the spectral asymptotic in the Large Coupling Limit
}
\author{Luka Grubi\v si\' c}
\address{
Department of Mathematics, University of Zagreb, Bijeni\v{c}ka cesta 30, 10000 Zagreb, Croatia}

\thanks{This work is based on a part of author's PhD thesis \cite{GruPhd}, which was written
under the supervision of Prof. Dr. Kre\v{s}imir Veseli\'{c}, Hagen in partial
fulfilment of the requirements for the degree Dr. rer. nat. The paper has been written at the time when the author was
a member of ``Institut f{\"u}r reine und angewandte Mathematik'',
RWTH Aachen University. It has been submitted for a review since February 2008.
}
\email{luka.grubisic@math.hr}
\date{\today}

\begin{document}
%
%
\def\ra{{\sf R}}
\def\je{{\sf N}}
\renewcommand{\Im}{{\ensuremath{\mathrm{Im\,}}}}
\renewcommand{\Re}{{\ensuremath{\mathrm{Re\,}}}}

\newcommand{\diag}{\mathrm{diag}}
\def\fr{\mathfrak{r}}
\def\fh{\mathfrak{h}}
\def\fa{\mathfrak{a}}
\def\fm{\mathfrak{m}}
\def\fk{\mathfrak{k}}
\def\fu{\mathfrak{u}}
\def\fx{\mathfrak{x}}
\def\fv{\mathfrak{v}}
\def\fw{\mathfrak{w}}
\def\fs{\mathfrak{s}}
\def\fG{\mathfrak{g}}
\def\fI{\mathfrak{I}}
\def\fE{\mathfrak{E}}
\def\fC{\mathfrak{C}}
\def\I{\mathbf{I}}
\def\mA{\mathbf{A}}
\def\mB{\mathbf{B}}
\def\mM{\mathbf{M}}
\def\mU{\mathbf{U}}
\def\mX{\mathbf{X}}
\def\mV{\mathbf{V}}
\def\mW{\mathbf{W}}
\def\mY{\mathbf{Y}}
\def\mK{\mathbf{K}}
\def\mR{\mathbf{R}}
\def\mQ{\mathbf{Q}}
\def\mT{\mathbf{T}}
\def\mS{\mathbf{S}}
\def\mH{\mathbf{H}}
\def\x{\mathcal{X}}
\def\q{\mathcal{Q}}
\def\b{\mathcal{B}}
\def\g{\mathcal{G}}
\def\d{\mathcal{D}}
\def\K{\mathcal{K}}
\def\H{\mathcal{H}}
\def\vp{\mathcal{V}}
\def\wp{\mathcal{W}}
\def\lp{\mathcal{L}}
\def\EE{\mathcal{E}}
\def\SS{\mathcal{S}}
\def\cT{\mathcal{T}}
\def\y{\mathcal{Y}}
\def\z{\mathcal{Z}}
\def\D{\mathcal{D}}
\def\R{ \mathbb{R}}
\def\C{ \mathbb{C}}
\def\N{ \mathbb{N}}
%
%
\newtheorem{theorem}{Theorem}[section]
\newtheorem{corollary}[theorem]{Corollary}
\newtheorem{lemma}[theorem]{Lemma}
\newtheorem{proposition}[theorem]{Proposition}
\theoremstyle{definition}
\newtheorem{remark}[theorem]{Remark}
\newtheorem{assumption}[theorem]{{\rm \textbf{Assumption}}}
\newtheorem{example}[theorem]{Example}
\newtheorem{problem}{{\rm \textbf{Problem}}}[section]
\newtheorem{definition}[theorem]{Definition}
\numberwithin{equation}{section}
%
%
\newcommand{\norm}[1]{\Vert #1 \Vert}
\newcommand{\normm}[1]{\Vert #1 \Vert}
\newcommand{\ugl}[1]{\left[ #1 \right]}
\newcommand{\sk}[1]{\left( #1 \right)}
\newcommand{\dual}[1]{\left< #1 \right>}
\newcommand{\abs}[1]{\left| #1 \right|}
\newcommand{\abss}[1]{\vert #1 \vert}
\newcommand{\spa}[1]{\mbox{span}\{ #1 \}}
\newcommand{\absf}[2]{\frac{\abss{#1}}{\abss{#2}} }
\newcommand{\ine}[1]{{\mathbf{#1}}}
\newcommand{\conpo}[1]{\stackrel{#1}{\longrightarrow}}
%
%
\def\imag{{\rm i}}
\def\sinbf{\mathsf{sin}}
\def\cosbf{\mathsf{cos}}
\def\eexp{\text{e}}
\def\region{\mathcal{R}}
\def\wlim{\text{{\rm w-lim}}}
\def\tripleb{\mid\!\mid\!\mid}

\begin{abstract}
We prove optimal convergence estimates for eigenvalues
and eigenvectors of a class of singular/stiff perturbed problems. Our profs are constructive
in nature and use (elementary) techniques which are of current interest in
computational Linear Algebra to obtain estimates even
for eigenvalues which are in gaps of the essential spectrum. Further,
we also identify a class of ``regular'' stiff perturbations with (provably)
good asymptotic properties. The Arch Model from the theory of elasticity is
presented as a prototype for this class
of perturbations. We also show that we are able to study model problems which do not satisfy this
regularity assumption by presenting a study of a Schroedinger operator with singular obstacle potential.
\end{abstract}

\maketitle

\section{Introduction}\label{s:prvo}
In this paper we give sharp estimates for the asymptotic behavior of the spectral problem for the family of self-adjoint operators
$\mH_\kappa$ which are defined by positive definite quadratic forms
\begin{equation}\label{eq:def}
\fh_\kappa(u, v)=\fh_b(u, v) +\kappa^2 \fh_e(u, v), \qquad u, v\in\q(\fh_b)\subset\q(\fh_e).
\end{equation}
Here we have used $\q(\fh_b)$ and $\q(\fh_e)$ to denote the domain of definition of  $\fh_b$ and $\fh_e$
and we assume that $\kappa^2\to\infty$. Qualitative results for
families of self-adjoint operators like $\mH_\kappa$ have a long tradition. We are
particularly influenced by the results from \cite{SanchezPalencia90,WeidmannScand84}.
Here by qualitative results we mean those results
which prove (e.g.) that the spectral projections $E_\kappa(\cdot)$, $\mH_\kappa=\int\lambda~dE_\kappa(\lambda)$ converge
in some appropriate sense.

To give a first idea of what is hidden within the abstract formulation (\ref{eq:def}) let us consider
two simple examples that are representative for more complex model problems (studied later on in
Section \ref{s:Arch}). The family of quadratic forms
\begin{equation}\label{e:heat}
\fh_\kappa(u,v)=\int_{0}^2 u'v'~dx +\kappa^2\int_{1}^2 u'v'~dx,
\quad u,v\in H^1_0[0,2],\;\;\kappa\to\infty
\end{equation}
is paradigmatic for a regularly perturbed family, whereas the family
\begin{equation}\label{e:GL}
\fh_\kappa(u,v)=\int_{0}^2 u'v'~dx +\kappa^2\int_{1}^2 uv~dx,
\quad u,v\in H^1_0[0,2],\;\;\kappa\to\infty
\end{equation}
is representative for the quadratic forms which violate our new regularity assumption. Note that
in our relative theory the unbounded perturbation $\fh_e$ in (\ref{e:heat}) is preferable to the
bounded perturbation $\fh_e$ in (\ref{e:GL}).
Here we have used $H^1_0(\cdot)$ to denote the standard Sobolev spaces.

The limit of the families like (\ref{e:heat}) and (\ref{e:GL}) can be a non-densely defined operator
and we use the theory of \cite{WeidmannScand84} to study the convergence of such $\fh_\kappa$ and
associated $\mH_\kappa$ as $\kappa\to\infty$.
Let now the operator $\mH_\infty$ (in general non-densely defined) be the limit
(in the sense of \cite{WeidmannScand84}) of $\mH_\kappa$ as $\kappa\to\infty$ .
We use $\lambda_i^\kappa$, $i\in\N$ to denote the discrete eigenvalues of
$\mH_\kappa$, which are below the
infimum of the essential spectrum and are ordered in the ascending order according to multi\-pli\-city.
By $v_i^\kappa\in\q(\fh_\kappa)$,
$\mH_\kappa v_i^\kappa=\lambda_i^\kappa v_i^\kappa$ and $\|v_i^\kappa\|=1$
we denote accompanying eigenvectors. Here we allow $\kappa>0$ or formally $\kappa=\infty$.
Using the perturbation techniques from \cite{Gru05_3,Gru03_3,GruVes02,Gru_Ves_Sylv}
we prove (among other results) in the case of regular family
of the type (\ref{eq:def}); for a definition see Section \ref{sec:reg} below; the estimates
\begin{align}\label{eq:1}
\frac{\text{lb}}{\kappa^2}\leq\frac{|\lambda_i^\kappa-\lambda_i^\infty|}{\lambda_i^\infty}&\leq\frac{\text{ub}_1}{\kappa^2}\\
\frac{\text{lb}}{\kappa^2}\leq\frac{\fh_\kappa[v_i^\infty-v_i^\kappa]}{\fh_\kappa[v_i^\kappa]}&\leq\frac{\text{ub}_2}{\kappa^2}\\
\|E_\kappa(D)-E_\infty(D)\|&\leq \frac{\text{ub}_3}{\kappa^2},\qquad D\in\R\setminus\spec(\mH_\infty),\label{eq:3}
\end{align}
and we compute the constants $\text{lb}$ and $\text{ub}_i$, $i=1,2,3$ explicitly for several
concrete model problems. Further, we also give a formula for determining a critical $\kappa_0$ such that
(\ref{eq:1})--(\ref{eq:3}) hold for $\kappa\geq\kappa_0$ and we show that the estimates are optimal in the sense
that
$
\lim_{\kappa\to\infty}\frac{|\lambda_i^\kappa-\lambda_i^\infty|}{\lambda_i^\infty}\Big(\frac{\text{lb}}{\kappa^2}\Big)^{-1}=1
$
holds.

To show that our abstract approach to problems (\ref{eq:def}) does not incur accuracy tradeoffs ---when applied to concrete problems---we
consider several case studies. A prototype for the (less trivial)
regular problem is the Arch Model from e.g. \cite[Chapter 8.8:3]{CiarletV4y78}.
In our case study we compute explicit estimates for the asymptotic behavior of the eigenvalues
and spectral projections of the low frequency problem as the diameter of the arch goes to zero.
The limit of such family of arches is the
so called Curved Rod Model from \cite{JurakTambaca,Tambaca1D}. On the other hand,
Schroedinger (like) operators from \cite{BaumDemuth,BruneauCarbou,DemuthJeskeKirsch} are
representative for (higher dimensional) operators which have less ``well-behaved'' spectral asymptotic.
More to the point, in the case of the Schroedinger (like) operator from (\ref{e:GL}) we obtain
the same optimality statements, but the convergence is of the fractional order
$O\big(\frac{1}{\kappa^{2\alpha}}\big)$, $\alpha=\frac{1}{2}$
(~cf. \cite{DemuthJeskeKirsch,Gru05_3} for higher dimensional problems in unbounded domains).
These concrete examples determine a framework for presenting our (otherwise) more abstract results.
\subsection{Local (resolvent) estimates}
We approach this analysis by reformulating the convergence problem so that the
perturbation framework and the error representation formulae (this is the main constructive feature of our framework)
from \cite{DrmHAri97,Gru05_3,Gru03_3,GruVes02,Gru_Ves_Sylv} can be applied as a backbone of our construction.
A difference between our approach and the standard results of works like
\cite{BruneauCarbou,Dancer,DemuthJeskeKirsch,Panasenko} can best be seen when considering a way to compute a
constant $\text{ub}_3$ for an estimate like (\ref{eq:3}). The standard approach requires a study of the
integral
\begin{equation}\label{eq:diffi}
\oint_{\fC(\lambda_i^\infty)}\Big[(\zeta-\mH_\infty)^{-1}P_{\je(\fh_e)}-(\zeta-\mH_\kappa)^{-1}\Big]~d\zeta,
\end{equation}
where $\fC(\lambda_i^\infty)$ is a circle in the resolvent set of $\mH_\kappa$ which has $\lambda_i^\infty$
in its interior and the rest of the spectrum in its exterior. This frequently leads to cumbersome estimation formulae.
Thanks to the local character of the error representation formula from \cite{Gru05_3},
we are able to base our theory on a study of the integrals\footnote{The notation
$(\cdot, \cdot)$ and $\|\cdot\|$ always refer to the scalar product and the norm of the
background Hilbert space $\H$. The functions of the operator like $\mH_e^{1/2}$ are always meant in the sense of
the spectral calculus. By $P_{\je(\fh_e)}$ we generically denote
the $\H$ orthogonal projection onto the space $\je(\fh_e)$}
\begin{equation}\label{eq:formulaD}
(v_i^\infty, \mH^{-1}_\kappa v_i^\infty)-(v_i^\infty, \mH^{-1}_\infty v_i^\infty)=
\int^\infty_{\kappa^2}\|\mH_e^{1/2}\mH_\tau^{-1}v_i^\infty\|^2~d\tau, \;\; i=1,\ldots,m.
\end{equation}
Here $\mH_e$ is the operator defined by $\fh_e$ in the sense of Kato and
$m\in\N$ is the multiplicity of $\lambda_i^\infty$. The results from \cite{Brasche,DemuthJeskeKirsch} show that the
integrals (\ref{eq:formulaD}) are better amenable for a quantitative study than are (\ref{eq:diffi}).

Due to the difficulties in dealing with a formula like (\ref{eq:diffi}),
typical results from semiclassical analysis from e.g. \cite{Dancer,SanchezPalencia90}
establish only the fact that the projections converge in a much weaker sense (than is the convergence
of spectral projections in norm)
without giving information on the speed of convergence as measured by the coupling $\kappa^2$. The nearest in
spirit to our analysis is the approach of \cite{Panasenko}. However, in this work only a particular family
of model problems is considered and no estimates for the convergence of $E_\kappa(\cdot)$
in (unitary invariant) operator norm(s) are presented. Furthermore, the authors do not discuss the
radius of convergence of their ``asymptotic''  expansions.
For the geometric theory on the relationship between two projections and the importance of
establishing convergence estimates for all unitary invariant operator norms we refer the reader to the
seminal works \cite{DavisKahan70,HalmosTwo}.

\subsection{A notion of regularity}\label{sec:reg}
Let us now make precise what we mean by the regularity of $\fh_e$. In the terminology of \cite{SanchezPalencia90} a family of the type (\ref{eq:def}) is said
to be \textit{non-inhibited stiff} if $\fh_e$ is a closed and positive quadratic form and the subspace
\begin{equation}\label{eq:stiff}
\je(\fh_e):=\{u\in\q(\fh_e)~:~\fh_e[u]:=\fh_e(u, u)=0\}
\end{equation}
(of $\H$) is nontrivial. For technical convenience we assume (~without reducing the level of the generality)
that $\fh_b$ is positive definite and
use $\mH_b$ and $\mH_e$ to denote the self-adjoint operators which are defined in the sense of Kato
by $\fh_b$ and $\fh_e$ respectively.

We identify the \textit{regular family} of quadratic forms---with structure (\ref{eq:def})---by requiring that $\fh_b$ and $\fh_e$
satisfy a Ladyzhenskaya--Babu\v{s}ka--Brezzi type condition
\begin{equation}\label{eq:BB}
\sup_{v\in \q(\fh_e)}\frac{|(q, \mH_e^{1/2}v)|}{\fh_b[v]^{1/2}}\geq\frac{1}{\fk}\|P_{\je(\fh_e)}q\|,\qquad q\in\H,
\end{equation}
for some $\fk$, $\fk>0$.
The condition (\ref{eq:BB}) is equivalent with the claim that $\ra(\mH_e^{1/2}\mH_b^{-1/2})$, the range
of the operator $\mH_e^{1/2}\mH_b^{-1/2}$, is closed in $\H$, cf. examples (\ref{e:heat}) and (\ref{e:GL}).
The ramifications of the assumption (\ref{eq:BB}) will enable us to formulate a new method for studying
integrals (\ref{eq:formulaD}) for this class of model problems and thus complement
the study of singular obstacle potentials from \cite{Brasche,DemuthJeskeKirsch}.


\subsection{An outline of the paper}
Let us finish the introduction by briefly outlining the structure of the paper. In Section \ref{section2}
we introduce the notation and present the qualitative convergence framework from \cite{WeidmannScand84}.
The main approximation results of the paper appear in Section \ref{section3}.
To be more precise in Section \ref{section31} we review the operator matrix approach to Ritz value estimation
from \cite{Gru_Ves_Sylv,Gru05_3}. In Section \ref{section32} this approach to spectral estimation is specialized
to the problems of the large coupling limit. In particular we make precise in which sense can these estimates
be considered sharp. We also revisit, in Section \ref{section321}, the example from \cite{Gru05_3}
to show how do (\ref{eq:1})--(\ref{eq:3}) look in praxis for a non-regular $\fh_e$.
In Section \ref{section4} we characterize regular perturbations $\fh_e$ and give convergence
estimates which utilize this additional structural information. In Section \ref{s:Arch} we consider a model problem from the elasticity theory
and show that its asymptotic behavior is regular. In the last section we put the results in the broader context
and give an outlook of further research.

At the end we would like to emphasize that our study is distinguished by its constructive character.
This can be seen in the fact that we
give a general method to compute the constants $\text{lb}$ and $\text{ub}_i$, $i=1,2,3$ (as functions
of $\mH_\kappa$ and $v^\infty_i$) in (\ref{eq:1})--(\ref{eq:3}). With such a result we give a method to establish both a first order correction
for the limit eigenvalue $\lambda_i^\infty$, as well as to assess
the quality of this approximation to $\lambda_i^\kappa$. The optimality result
is a justification of this claim.
For other connections between
the elementary linear algebra and spectral theory we refer the reader to \cite{sjoestrand-2003}.

\section{Convergence of non-densely defined quadratic forms}\label{section2}

In this section we fix the notation and give background information on the previous results which we use.
We follow the general notational conventions and the terminology of
Kato \cite[Chapters VI--VIII]{Kato76}. Minor differences are contained in the following list
of notation and terminology.
%
%
%
%
\begin{itemize}
\item $\H$ ... is an infinite dimensional Hilbert space, can be both real
  or complex
\item $(\cdot , \cdot)$; $\|\cdot\|$ ... the scalar product on $\H$, linear in the second
  argument and anti-linear (when $\H$ is complex) in the first; the norm on $\H$
\item $\H_1\oplus\H_2$... the direct sum of the Hilbert spaces $\H_1$ and $\H_2$, for any $x\in\H_1\oplus\H_2$
we have $x=x_1\oplus x_2=\begin{bmatrix}x_1\\x_2\end{bmatrix}$ for $x_i\in\H_i$, $i=1,2$
\item $\spec(\mH)$, $\spec_{ess}(\mH)$; $\les(\mH)$ ... the spectrum and the essential
  spectrum of $\mH$; the infimum of the essential spectrum of $\mH$
\item $A\leq B$ ... order relation between self-adjoint operators (matrices),
is equivalent with the statement that $B-A$ is positive
\item
$\lp(\H)$; $\lp(\H_1,\H_2)$... the space of bounded linear operators
on $\H$, which is equipped with the norm $\|\cdot\|$; the space of bounded linear
operators from $\H_1$ to $\H_2$
\item $\textsf{R}(X), \textsf{N}(X)$ ... the range and the null space of the linear operator $X$
\item $\mA^\dagger$ ... the generalized inverse of the closed densely
defined operator $\mA$. If $\mA$ has the closed range then
$\mA^\dagger=(\mA(\mA^*\mA)^{-1})^*$ is bounded, see \cite{Nashed76}. We will
extend this notion below to hold for non-densely defined
self-adjoint operators.
\item $P$, $P_\perp$... the orthogonal projections $P$ and $P_\perp:=\I-P$
\item $j_{(\cdot)}$ ... a permutation of $\N$
\item $\diag(M,W)$ ... the block diagonal operator matrix
with the operators $M,W$ on its diagonal. The operators $M,W$ can be
both bounded and unbounded. The same notation is used to define the
diagonal $m\times m$ matrix\\ $\diag(\alpha_1,\cdots,\alpha_m)$,
with $\alpha_1,\cdots,\alpha_m$ on its diagonal.
\item $s_1(A)\geq s_2(A)\geq\cdots$, $s_{\max}(A), s_{\min}(A)$ ...
the singular values of the compact operator $A$ ordered in the
descending order according to multiplicity, the minimal (if it
exists) and the maximal singular value of $A$
\item $\tripleb X \tripleb$ ... a unitary invariant or operator cross norm
of the operator $X$. Since $\tripleb\cdot\tripleb$ depends
only on the singular values of the operator, we do not notationaly
distinguish between the instances of the norm
$\tripleb\cdot\tripleb$ on $\lp(\H)$, $\lp(\ra(P))$,
$\lp(\ra(P),\ra(P)^\perp)$, or such. For details see \cite{SimonTrace}.
\item $\tr(X)$, $\tripleb X\tripleb_{HS}$ ... the trace a the Hilbert--Schmidt
norm of the operator $X$, it holds $\tripleb X\tripleb_{HS}=\sqrt{\tr(X^*X)}$, see \cite{SimonTrace}
\end{itemize}
As a general policy to simplify the notation we shall always drop
indices  when there in no danger of confusion.

Let us assume that we have a closed, symmetric and semibounded from below form
$\fh$ with the dense domain $\q(\fh)\subset\H$ as given
in \cite[(VI.1.5)--(VI.1.11), pp. 308--310]{Kato76}.
The form $\fh$ which has a strictly positive lower bound will be called
\textit{positive-definite}. This is also a small departure from the terminology
of \cite[Section VI.2, pp. 310]{Kato76}.
Such $\fh$ defines the self-adjoint and
positive definite operator $\mH$ in the sense of \cite[Theorem VI.2.23, pp. 331]{Kato76}.
Furthermore, the operator
$\mH$ is densely defined with the domain $\d(\mH)\subset\q(\fh)$ and $\d(\mH^{1/2})=\q(\fh)$.
We also generically assume that $\mH$ has discrete eigenvalues
$
\lambda_1(\mH)\leq\cdots\leq\lambda_m(\mH)\leq\cdots<\les(\mH),
$
where we count the eigenvalues according to multiplicity. Another departure from the terminology of Kato is that
we use $\fh(\psi, \phi)$ to denote the value of $\fh$ on $\psi, \phi \in\q(\fh)$, but we write $\fh[\psi]
:=\fh(\psi, \psi)$
for the associated \textit{quadratic form} $\fh[\cdot]$. We also
emphasize that we use $\cdot^*$ to denote the adjoint both in
the real as well as in the complex Hilbert space $\H$ as is customary in \cite[Chapters VI--VIII]{Kato76}.

In order to be able to handle the problems of the type (\ref{eq:def}), we shall need
to work with operators that are not necessarily densely defined,
cf. (\ref{e:heat}) and (\ref{e:GL}). We use the notion of the \textit{pseudo inverse}
of the operator $\mH$ that is assumed to be self-adjoint in the closure of its domain of definition
$\overline{\d(\mH)}^{~_{\|\cdot\|}}\subset\H$ (tacitly assumed to be a non-trivial subspace).
A definition from \cite{WeidmannScand84} will be used. The
\textit{pseudo inverse}
of the operator $\mH$ is the self-adjoint operator $\widehat{\mH}$ defined by
\begin{align*}
\d(\widehat{\mH})&=\ra(\mH)\oplus\d(\mH)^\perp,\\
\widehat{\mH}(u+v)&=\mH^{-1}u,\qquad u\in\ra(\mH),~v\in\d(\mH)^\perp.
\end{align*}
It follows that $\widehat{\mH}=\mH^{-1}$ in $\overline{\ra(\mH)}^{~_{\|\cdot\|}}$ and $\widehat{\mH}$ is
bounded if
and only if $\ra(\mH)$ is closed in $\H$.
When considered solely in $\overline{\d(\mH)}^{~_{\|\cdot\|}}$ the operator
$\mH$ is obviously self-adjoint, so we can also use the spectral calculus from \cite{Simon78-canonical}
to define the \textit{generalized inverse}, which extends the definition from the case of the densely defined operator,
 as
\begin{align*}
\mH^{\dagger}&=f(\mH),\qquad f(\lambda)=\begin{cases}0,&\lambda=0\\\frac{1}{\lambda},&\lambda > 0\end{cases}\\
\d(\mH^{\dagger})&=\{u\in\H~:~\int f^2(\lambda)d (E(\lambda)u, u)<\infty\},
\end{align*}
where $E(\cdot)= E_{\mH}(\cdot)P_{\d(\mH)}$. Obviously, we have $\d(\widehat{\mH})\oplus\je(\mH)=\d(\mH^{\dagger})$
and the identity $\mH^\dagger u=\widehat{\mH}u$, $u\in\d(\widehat{\mH}^{1/2})$ holds. In further text we shall tacitly drop
the notational distinction between the generalized and pseudo inverse.
The usual monotonicity properties can be extended to
the generalized inverse. In particular it holds
\begin{equation}\label{eq:order}
\|\mH^{1/2}_1\!u\|\leq\|\mH^{1/2}_2\!u\|, \;\; u\in\d(\mH_2^{1/2})
\Leftrightarrow\|\mH^{1/2\dagger}_2\!u\|\leq\|\mH^{1/2\dagger}_1\!u\|,\;\; u\in\d(\widehat{\mH}^{1/2}_1).
\end{equation}
This monotonicity principle is the main ingredient of the proof of the convergence result for (\ref{eq:def}). When dealing with non-densely defined forms
this principle can be formulated as follows. Let $\fh_1$ and $\fh_2$ be two closed positive definite forms and let
$\mH_1$ and $\mH_2$ be the self-adjoint operators defined by $\fh_1$ and $\fh_2$ in $\overline{\q(\fh_1)}$
and $\overline{\q(\fh_2)}$. We say $\fh_1\leq \fh_2$ when $\q(\fh_2)\subset\q(\fh_1)$ and
\begin{equation}\label{e_uredjaj}
\fh_1[u]=\|\mH_1^{1/2}u\|^2\leq \fh_2[u]=\|\mH_2^{1/2}u\|^2,\qquad u\in\q(\fh_2).
\end{equation}
Equivalently, we write $\mH_1\leq\mH_2$ when $\fh_1\leq \fh_2$. Now, we can write the fact (\ref{eq:order})
as
\begin{equation}\label{e_radniuredjaj}
\mH_1\leq\mH_2\Longleftrightarrow\mH_2^{\dagger}\leq\mH^{\dagger}_1.
\end{equation}Let us define, for non-inhibited (see definition (\ref{eq:stiff})) quadratic forms like $\fh_\kappa$ from (\ref{eq:def}), the domain
$
\q_\infty:=\{u\in\q~:~\lim_{\kappa\to\infty} \fh_\kappa[u]<\infty\}
$,
then according to \cite{Simon78-canonical,WeidmannScand84} the symmetric form
$$
\fh_\infty(u, v)=\lim_{\kappa\to\infty}\fh_\kappa(u, v), \qquad u,v\in\q_\infty
$$
is closed in $\overline{\q_\infty}^{~_{\|\cdot\|}}$ and it defines the self-adjoint operator $\mH_\infty$ there.
Further, it holds that
$
\mH_\infty^\dagger=\slim_{k\to\infty}\mH_\kappa^{-1}.
$
The general framework for a description of families of converging
positive definite forms will be the following theorem from \cite{WeidmannScand84}.

\begin{theorem}\label{t:WCon}
Let $\fs_n$, $\fh_n$, $\fu_n$ and $\fh_\infty$ be closed symmetric forms in $\H$ such that they are all
uniformly\footnote{By this we mean that they have a uniform positive lower bound.} positive definite.
\begin{enumerate}
\item If $\fs_n\geq \fs_{n+1}\geq \fh_\infty$ where
\begin{align*}
\fh_\infty(u,v)&=\lim_{n\to\infty} \fs_n(u,v),\qquad u,v\in\bigcup_{n\in\N}\q(\fs_n)
\end{align*}
then $\fh_\infty$ is closed with $\q(\fh_\infty)=
\overline{\bigcup_{n\in\N}\q(\fs_n)}^{~_{\fh_\infty}}$ and $\mH_\infty^\dagger=\slim_n\mS_n^\dagger$.
\item If $\fu_n\leq \fu_{n+1}\leq \fh_\infty$ where
\begin{align*}
\fh_\infty(u,v)&=\lim_{n\to\infty} \fu_n(u,v),\qquad u,v\in\q(\fh_\infty)
\end{align*}
then $\fh_\infty$ is closed with $\q(\fh_\infty)=
\left\{f\in\bigcap_{n\in\N}\q(\fu_n)~:~\sup \fu_n[f]<\infty\right\}$
 and $\mH_\infty^\dagger=\slim_n\mU_n^\dagger$.
\item If $\fu_n$ and $\fs_n$ are as before and $\fu_n\leq \fh_n\leq \fs_n$
also holds, then
\begin{align*}
\fh_\infty(u,v)&=\lim_{n\to\infty} \fh_n(u,v),\qquad u,v\in\q(\fh_\infty),\\
\mH_\infty^\dagger&=\slim_{\kappa\to\infty}\mH^\dagger_n.
\end{align*}
\end{enumerate}
\end{theorem}
For the families of forms which satisfy the assumptions of Theorem \ref{t:WCon} the following
qualitative convergence result on spectral families has been established in \cite{WeidmannScand84}.

\begin{theorem}\label{t:WeNorm}
Let $\fh_n$ be a sequence of positive definite forms that satisfies any of the
assumptions of Theorem \ref{t:WCon} for $\fu_n$,  $\fs_n$ or $\fh_n$. Let there also be the
positive definite form $\fs$ such that $\fh_n\geq \fs$ and $\lambda_e(\mS)>0$.
Then
\begin{equation}\label{eq:WeiQual}
\|E_{n}(D)-E_{\infty}(D)\|\to 0, \qquad D<\lambda_e(\mS), D\not\in\spec(\mH_\infty).
\end{equation}
\end{theorem}

The results like Theorem \ref{t:WCon} have independently been obtained in
\cite{SanchezPalencia90,Simon78-canonical}. We have opted for Theorem \ref{t:WCon} since
it extensively uses the monotonicity (or ``sandwiched'' monotonicity) to establish the
stability of the converging eigenvalues and this fits neatly into the perturbation framework of
\cite{DrmHAri97}. This was the chief source of motivation for
the main construction from the PhD thesis \cite{GruPhd} (those results appeared later
in \cite{Gru05_3,Gru03_3,GruVes02,Gru_Ves_Sylv}).

\section{A constructive approach to asymptotic eigenvalue/eigenvector estimates}\label{section3}

Let us reiterate that we use the notion of the constructiveness in this paper in two contexts. First, it should emphasize
that all of our theory is bases on the error representation result like (\ref{eq:apply})--(\ref{eq:apply2}), below.
But second, it is also meant to emphasize that in a result like those of the type
(\ref{eq:1})--(\ref{eq:3}) we present a way to construct an improvement to the approximation $\lambda_i^\infty$ (of
the eigenvalue $\lambda_i^\kappa$).
The constants $\text{lb}$ and $\text{ub}_i$, $i=1,2,3$ are explicit functions of the approximation
defects $\eta_i(P)$, to be defined below and it is the aim of this section to reveal this
dependence.

\subsection{Background information on the block-diagonal part of the operator/form}\label{section31}
In this section we review the results form our previous work which we use to
prove our first contribution in Section \ref{section32}. A reader who would like
to go straight to the new results can do that directly after reading equation
(\ref{eq:block}) and Definition \ref{def:kapprox} below.

In this section we assume that we have a fixed closed symmetric and densely defined form $\fh$.
We will review the basic spectral properties of the \textit{block-diagonal part} of $\fh$
with respect to orthogonal projection $P$, $\ra(P)\subset\q(\fh)$
as is presented in \cite{Gru03_3}. In order to simplify the presentation we temporarily
suppress (in the notation) the dependence of quantities on $\mH$, where there is no danger
of confusion.
Assuming that $\ra(P)$ is finite dimensional we define the \textit{block-diagonal} part of $\fh$
by setting
\begin{equation}\label{eq:block}
\fh_P(u,v):=\fh(Pu, Pv)+\fh(P_\perp u, P_\perp v),\qquad u, v \in\q(\fh_P):=\q(\fh).
\end{equation}
Obviously the form $\fh_P$ is closed and positive definite and so it defines the
 self-adjoint operator $\mH_P$ in the sense of Kato. We further have (for a proof see
 \cite{Gru05_3,Gru03_3}):
\begin{align}\label{eq:cons1}
&\ra(\mH^{-1}-\mH_P^{-1})\qquad  \textrm{is finite dimensional}.\\
&\eta_{\max}(P):=\sup_{u\in\q(\fh)}\frac{|\fh[u]-\fh_P[u]|}{\fh_P[u]}<1.
\label{eq:cons2}
\end{align}
A first consequence of these two features is the stability of essential spectra, namely Weyl's theorem gives
$\spec_{{\rm ess}}(\mH)=\spec_{{\rm ess}}(\mH_P)$. Further, we have the estimate---of the same form as (\ref{eq:cons2})---
for the eigenvalues $\lambda_i(\mH_P)$ and $\lambda_i(\mH)$, $i\in\N$ which are below the infimum
of the essential spectrum $\lambda_{{\rm ess}}(\mH)=\lambda_{{\rm ess}}(\mH_P)$
\begin{equation}\label{eq:rem}
\frac{|\lambda_i(\mH)-\lambda_i(\mH_P)|}{\lambda_i(\mH_P)}<\eta_{\max}(P),\qquad i\in\N.
\end{equation}
The attractiveness of interpreting the form $\fh$ as a perturbation of its block-diagonal part lies in the fact
that
\begin{equation}\label{eq:local}
\spec(\mH_P)=\spec(\Xi)\cup\spec(\mW)
\end{equation}
where $\Xi=(\mH^{1/2}P)^*(\mH^{1/2}P)\big|_{\ra(P)}$ is a finite dimensional
 operator and $\mW$ is the self-adjoint operator which is defined
in $\ra(P_\perp)$ by the quadratic form $\fh(P_\perp\cdot, P_\perp\cdot)$. Since
 $\spec(\Xi)$ is computable, we can start building our constructive estimation procedure on this fact.
 As a convention we will use $\mu_1\leq\cdots\leq\mu_{\dim\ra(P)}$ to denote the eigenvalues of $\Xi$. The numbers $\mu_i$
 will be called the \textit{Ritz values} from the subspace $\ra(P)$. In this section we also use the notation
 $\lambda_i:=\lambda_i(\mH)$.

Let us now assume that $\dim \ra(P)=m\in\N$. To examine the relationship between $\fh$ and $\fh_P$ in further detail define
\begin{equation}\label{eq.sing_val_2}
\eta_i(P):=\Big[\!\!\!\max_{\substack{\mathcal{S}\subset\ra(P),\\
\dim(\mathcal{S})=m-i+1}}\!\!\!\!\min\big\{\frac{(\psi,
\mH^{-1}\psi)-(\psi,\mH^{-1}_P\psi)}{(\psi,\mH^{-1}\psi)}~
\big|~\psi\in\mathcal{S}, \|\psi\|=1\big\}\Big]^{1/2},
\end{equation}
for $i=1, \ldots, m$. It has also been shown in \cite{Gru03_3} that $\eta_{\max}(P)=\eta_m(P)$.
Although the perturbation $\delta_P(\fh):=\fh-\fh_P$
is in general---for some
$P$, $\ra(P)\subset\q(\fh)$---not representable by an operator, the quadratic form
$\delta_P^s(\fh)[\cdot]:=\fh[\mH^{-1/2}_P\cdot]-\fh_P[\mH^{-1/2}_P\cdot]$ can always be represented by
the bounded operator block-matrix (with respect to $P\oplus P_\perp=I$)
$$
\delta_P^s(H)=\begin{pmatrix}
0&\Gamma^*\\\Gamma&0\end{pmatrix},\qquad \textrm{and}\quad(\cdot,\delta_P^s(H)\cdot)=\delta_P^s(\fh)[\cdot].
$$
Furthermore, \cite[Lemma 2.1]{Gru05_3} gives that $s_i(\Gamma)=\eta_i(P)$, $i=1,\ldots,m$.
 The analysis of \cite{Gru05_3} now yields the conclusion that the
test space $\ra(P)$ can be used to generate good approximation for the
eigenvalues $\lambda_i$, $i=q, \ldots, q+m-1$ when $\eta_m(P)$ is smaller than half of
the \textit{relative gap} $$\gamma_q:=\min\Big\{
\frac{\lambda_{q+m}-\mu_m}{\lambda_{q+m}+\mu_m},
\frac{\mu_1-\lambda_{q-1}}{\mu_1+\lambda_{q-1}}\Big\}.$$
 Ample numerical evidence
corroborate that such estimates are robust (with regard to scaling) and sharp. Assume
that $\eta_{\max}(P)<\frac{1}{2}\gamma_q$ and that $\dim\ra(P)=m$, where $m$ is
the multiplicity of the eigenvalue $\lambda_q$. Using \cite[Theorem 3.3]{Gru05_3} we conclude that the operator matrix
\begin{equation}\label{eq:Schur}
\delta_P^s(H_q)=\left[\begin{matrix}\I-\lambda_q\Xi^{-1}&
\Gamma^*\\ \Gamma&\I- \lambda_q\mW^{-1} \end{matrix}\right],
\end{equation}
which is the block-matrix representation (with respect to $P\oplus P_\perp=\I$) of the quadratic form
$$
\delta_P^s(\fh_q)[\cdot]:=\fh(\mH^{-1/2}_{P}\cdot,\mH^{-1/2}_{P}\cdot)-\lambda_q(\mH^{-1/2}_{P}\cdot,
\mH^{-1/2}_{P}\cdot),
$$
satisfies $\dim\je(\delta_P^s(H_q))=m$ and the mechanism of \cite[(1.1)--(1.2)]{sjoestrand-2003}---also known in the
Linear Algebra as the
Wilkinson's Schur complement trick (see \cite[pp. 183]{Parlett80} and \cite[Theorem 3.3]{Gru05_3})---allows us to conclude
\begin{align}\label{eq:apply}
\I-\lambda_q\Xi^{-1}&=\Gamma^*(\I-\lambda_q\mW^{-1})^{-1}\Gamma\\
&=\Gamma^*\Gamma+\lambda_q\Gamma^*\mW^{-1/2}(\I-\lambda_q\mW^{-1})^{-1}\mW^{-1/2}\Gamma.
\label{eq:apply2}
\end{align}
Identity (\ref{eq:apply}) is the basis of the proof of \cite[Theorem 3.3]{Gru05_3} which we now quote. Note
that (\ref{eq:apply})--(\ref{eq:apply2}) also hold for $\lambda_q$ which is in a gap of the essential
spectrum.
Based on the definition (\ref{eq.sing_val_2}) we now define (for later usage) the \textit{appro\-xi\-mation-defects}
for $\fh_\kappa$.
\begin{definition}\label{def:kapprox}
 Let the sequence $\fh_\kappa$ be given and let
the orthogonal projection $P$ be such that $\ra(P)\subset\q(\fh_\kappa)$ and $\dim\ra(P)<\infty$.
We write $\eta_i(\kappa,P)$ for $\eta_i(P)$ from (\ref{eq.sing_val_2}) when applied on $\fh_\kappa$.
We call $\eta_i(\kappa,P)$ the \textit{$\kappa$-approximation defects}. If we are given a subspace $\mathfrak{P}=\ra(P)$,
then we abuse (simplify) the notation
and freely write $\eta_i(\kappa,\mathfrak{P})=\eta_i(\kappa,P)$.
\end{definition}
\begin{theorem}\label{thm:second}
  Let the discrete eigenvalues of the positive definite operator $\mH$
  be so ordered that $ \lambda_{q-1}<\lambda_q=\lambda_{q+m-1}<
  \lambda_{q+m} $.  Let $\ra(P)\subset\q(h)$ be the test subspace such that
  $\dim\ra(P)=m$ and $ \frac{\eta_m(P)}{1-\eta_m(P)}<\gamma_q$.  Then we have
\begin{align}\label{trece:e_tap4b} \tripleb{\rm
  diag}(\frac{|\lambda_q-\mu_i|}{\mu_i})_{i=1}^m\tripleb&\leq
  \frac{\eta_m(P)}{\fG_{q,\eta_m(P)}}\tripleb{\rm diag}(\eta_i(P))_{i=1}^m\tripleb.
\end{align}
where $\fG_{q,\zeta}:=\max\big\{
  \frac
  {\mu_1(1-\zeta)-(1+\frac{\zeta}{1-\zeta})\lambda_{q-1}}{(1+\frac{\zeta}{1-\zeta})\lambda_{q-1}},
  \frac{(1-\frac{\zeta}{1-\zeta})\lambda_{q+m}-(1+\zeta)\mu_m}{
  (1-\frac{\zeta}{1-\zeta})\lambda_{q+m}}\big\}$
  for $q>1$ and we set
  $\fG_{1,\zeta}:=\fG_{1}:=\frac{\lambda_{m+1}-\mu_m}{\lambda_{m+1}+\mu_m}$.
  Here we use $\diag(\alpha_i)_{i=1}^m$ to denote the $m\times m$ diagonal
  matrix with scalars $\alpha_i$ on its diagonal and $\tripleb\cdot\tripleb$
  denotes any unitary invariant matrix norm and $\mu_i$ are the Ritz values from $\ra(P)$.
\end{theorem}

In the case in which we do not have explicit information on the
multiplicity of $\lambda_q$ we have a weaker upper estimate. There
is also an accompanying lower estimate which establishes the
equivalence of the estimators $\eta_i(P)$ and the error. Assuming that $\mH=\int\lambda~dE(\lambda)$ and
that we use
$v_i$, $\mH v_i=\lambda_i  v_i$, $\|v_i\|=1$ to denote eigenvectors and $\psi_i\in\ra(P)$, $\Xi \psi_i=\mu_i \psi_i$,
$\|\psi_i\|=1$ to denote Ritz vectors, we collect some representative spectral estimates (bases on $\ra(P)$)
  from \cite{Gru05_3,Gru_Ves_Sylv}.
\begin{theorem}\label{tm:ess}
  Let the discrete eigenvalues of the positive definite operator $\mH$
  be so ordered that $\lambda_{m}<\lambda_{m+1}$ and let
  $\lambda_{s_1}<\lambda_{s_2}<\cdots<\lambda_{s_p}$ be all the
  elements\footnote{We assume that $1\leq s_1<s_2<\cdots< s_p\leq m$.}
  of
  $\spec(\mH)\setminus\{\lambda\in\spec(\mH)~:~\lambda\geq\lambda_{m+1}\}$.
  If
  $\frac{\eta_m(P)}{1-\eta_m(P)}<\frac{\lambda_{m+1}-\mu_m}{\lambda_{m}+\mu_m}$
  then there exist eigenvectors $v_i$, $\mH v_i=\lambda_i  v_i$, $\|v_i\|=1$
  and Ritz vectors $\psi_i\in\ra(P)$, $\Xi \psi_i=\mu_i \psi_i$,
$\|\psi_i\|=1$ such that
\begin{align}\label{eq:Sylv}
\tripleb E(\mu_m)-P\tripleb
&\leq\frac{\sqrt{\lambda_{m+1}\mu_m}}{\lambda_{m+1}-\mu_m}~
\frac{\tripleb\diag((\eta_i(P))_{i=1}^m)\oplus \diag((\eta_i(P))_{i=1}^m)\tripleb}{\sqrt{1-\eta_m(P)}},\\
\label{trece:e_tap6}
  \frac{\mu_1}{2\mu_m}\sum_{i=1}^m\eta_i^2(P)&\leq\sum^m_{i=1}\frac{|\lambda_i-\mu_i|}{\mu_i}\leq
    \frac{1}{{\displaystyle
    \min_{i=1,\ldots,p}\fG_{s_i,\eta_{m_i}(P_{s_i})}}} \sum_{i=1}^m\eta_i^2(P),\\
\label {eq:gruves}
      \|v_i-\psi_i\|&\leq\max_{\lambda\in\spec(\mH)\setminus\{\lambda_i\}}\frac{\sqrt{2\lambda\mu_i}}
  {|\lambda-\mu_i|}\frac{\eta_m(P)}{\sqrt{1-\eta_m(P)}}, \\
  \frac{\fh[\psi_i-v_i]}{\fh[ v_i]}&=
  \|v_i-\psi_i\|^2+\frac{\mu_i-\lambda_i}{\lambda_i},\qquad
  i=1,...,m.\label{eq:strang}
\end{align}
Here $P_{s_i}$ is the orthogonal projection onto the linear span of
$\{\psi_{j}~:~j=\sum^i_{k=1}m_k+1,
\ldots,\sum^{i+1}_{k=1}m_k \}$ and $m_i$ is the multiplicity of the
eigenvalue $\lambda_{s_i}$, $i=1, \ldots,p$. Obviously the
identity $P_{s_1}\oplus P_{s_2}\oplus\cdots\oplus P_{s_p}=P$
holds. In the case in which $\lambda_1=\lambda_m$ we can drop the
constant $\frac{\mu_1}{2\mu_m}$ from the lower estimate. We can also allow for other cross norms
$\tripleb\cdot\tripleb$ of the
diagonal matrix $\diag((\eta_i(P))_{i=1}^m)$ in (\ref{trece:e_tap6}).
\end{theorem}
The proof of the estimate for the spectral projection (\ref{eq:Sylv}) can be found in
\cite{Gru_Ves_Sylv}, the proof of (\ref{eq:gruves}) is in \cite{Gru05_3} and
identity (\ref{eq:strang}) is well-known. For reader's convenience let us also point out that the problem of
estimating the spectral projections $E(\fI)$---where $\fI$ is some contiguous interval
whose boundary points are not the accumulation points of $\spec(\mH)$---
can be seen
as problem in obtaining a robust computable estimate of the Cauchy integral
\begin{equation}\label{eq:resE}
\|E(\fI)-P\|=\frac{1}{2\pi}\|\oint_{\fC(\fI)}(\zeta-\mH_P)^{-1}-(\zeta-\mH)^{-1}~d\zeta\|.
\end{equation}
By $\fC(\fI)$ we denote the circle in the resolvent set of $\mH$ such that $\fI$ is in the interior
of the associated disc and the rest of the spectrum is outside the disc.
However, contrary to the intuition, the direct analysis of (\ref{eq:resE}) is not
the most natural way to obtain computable
and robust estimates of $\|E(\fI)-P\|$. A problem is that, although the integral of the resolvent
difference does not depend on the integration path $\fC(\fI)$, estimates of it do. Furthermore, the
circle is only one of many possible curves which should be taken into account.
As an alternative we consider the approach of the
(weakly formulated) operator equations. Not only are the estimation formula which are so obtained
sharp (see \cite[Remark 2.3]{Gru_Ves_Sylv}), but also the technique allows for a natural consideration of
estimates which utilize other operator cross norms $\tripleb\cdot\tripleb$.
Such results are known as $\sin\Theta$ theorems in the recognition
of the milestone work \cite{DavisKahan70} and have been extensively studied in the
computational Linear Algebra, see \cite{Li-II-99,Ren-CangStructured} and the references there.
We use a recent generalization of those results,
which is particularly suitable
for an application in the quadratic form setting, see \cite{Gru_Ves_Sylv}.

\begin{remark}\label{rem:rel_gap}
  Note that as $\eta_{s_i}(P_i)\to 0$ we have
  $\fG_{s_i,\eta_{m_i}(P_i)}\!\!\to\min
\{\!\!\frac{\lambda_{s_{i+1}}-\lambda_{s_{i}}}{\lambda_{s_{i}}},
\frac{\lambda_{s_{i}}-\lambda_{s_{i-1}}}{\lambda_{s_{i-1}}}\!\}$ and ${
\min\{\fG_{s_i,\eta_{m_i}(P_i)}}~:~i=1,\ldots,p\}$ quantifies the minimal
\textit{relative} gap among the eigenvalues
$\lambda_{s_1}<\lambda_{s_2}<\cdots<\lambda_{s_p}$. Note that the
relative gap $\fG_{s_i,\eta_{s_i}(P_i)}$ distinguishes better between
the close eigenvalues than the \textit{absolute} gap, eg.
$\min\{\lambda_{s_{i+1}}-\lambda_{s_{i}},
\lambda_{s_{i}}-\lambda_{s_{i-1}}\}$ is an example of an absolute gap.
In Theorem \ref{tm:ess}, equivalently as in \cite[Proposition
2.3]{DrmacVeselic2}, we have that when
$\eta_{m_i}(P_i)<\frac{1}{3}\min_{k\ne
j}\frac{|\lambda_{s_k}-\lambda_{s_j}|}{\lambda_{s_k}+\lambda_{s_j}}$,
$i=1, \ldots,p$ then
  $$\frac{1}{{\displaystyle
\min_{i=1,\ldots,p}\fG_{s_i,\eta_{m_i}(P_i)}}}\leq
\frac{3}{{ \min_{k\ne
j}\frac{|\lambda_{s_k}-\lambda_{s_j}|}{\lambda_{s_k}+\lambda_{s_j}}}}.
$$
\end{remark}

\subsection{Estimates for the spectral asymptotic}\label{section32} We will now use
Theorem \ref{tm:ess}
to obtain convergence rate estimates for (\ref{eq:WeiQual}). This
is the central result which guaranties the stability of the spectrum of the
converging family of forms $\fh_\kappa$. Subsequently we will also prove results
like (\ref{eq:1})--(\ref{eq:3}) and use the motivating example of the
Schroedinger operator with a singular obstacle potential from \cite[Section 4]{Gru05_3} to
show our estimates in action.

Although we are working under the assumptions of Theorem \ref{t:WCon}, we assume---in order to be more explicit---
that we have the non-inhibited stiff family $\fh_\kappa$ from (\ref{eq:def}). The form $\fh_\infty$ obviously
defines the self-adjoint operator $\mH_\infty$ in $\je(\fh_e)$. By $\mH_\infty=\int\lambda E_\infty(\lambda)$
we denote the spectral representation of $\mH_\infty$ in $\je(\fh_e)$. We identify $E_\infty(\cdot)$ with
$E_\infty(\cdot)P_{\je(\fh_e)}$ and write $\mH_\infty=\int\lambda E_\infty(\lambda)$ for the non-densely
defined---in the space $\H$---operator $\mH_\infty$. Let $\fI$ be a contiguous interval
in $\R$, then $\fE_\infty^\fI:=\ra (E_\infty(\fI))$ is a subspace of $\q:=\q(\fh_b)$.
Let now $\fI$ be such that
$\fE_\infty^\fI$ is finite dimensional, then $\kappa$-approximation defect is given by
\begin{equation}\label{eq:AppDef}
\eta_i(\kappa,\fE_\infty^\fI):=\Big[\!\!\!\!\!\!\!\!\max_{\substack{\mathcal{S}\subset\fE_\infty,\\
\dim(\mathcal{S})=m-i+1}}\!\!\!\!\!\!\!\!\!\!\!\min\big\{\!\frac{(\psi,
\mH^{-1}_\kappa\psi)-(\psi,\mH^{-1}_{\fE_\infty}\psi)}{(\psi,\mH^{-1}_\kappa\psi)}~
\big|~\psi\in\mathcal{S}, \|\psi\|=1\big\}\Big]^{1/2},
\end{equation}
where $\mH_{\fE_\infty^\fI}^{-1}:=(\mH^{\dagger}_\infty)_{E_\infty(\fI)}=(\mH^{-1}_\kappa)_{E_\infty(\fI)}$ and $i=1, \ldots, \dim\ra(\fE_\infty^\fI)$.
To further simplify the notation we set $\eta_i(\kappa,\fI):=\eta_i(\kappa,\fE_\infty^\fI)$. Theorem \ref{t:WCon} now obviously yields
$$
\lim_{\kappa\to\infty}\eta_i(\kappa,\fI)=0,\qquad i=1,\ldots,\dim\ra(\fE_\infty^\fI).
$$
Similar construction can be performed in the case in which $\fE_\infty^\fI$ is infinite dimensional.
The main features which
are lost in this generalization are the easy computability of
 $\spec(\Xi^{-1}_\kappa)=\spec(\mH^{-1}_{\fE_\infty^\fI}E_\infty(\fI))$, the property that always
 $\eta_{\max}(\kappa,\fI)<1$
and the result on the stability of the essential spectrum. This makes, in general, such method less attractive for
practical constructive considerations.

Let us first give a quantitative version of Theorem \ref{t:WCon} which is based on the
application of Theorem \ref{tm:ess}. As a notational convenience we use
$\lambda_1^\infty\leq\cdots\leq\lambda_{i}^\infty\leq\lambda_{{\rm ess}}^\infty$ and
$\lambda_1^\kappa\leq\cdots\leq\lambda_{i}^\kappa\leq\lambda_{{\rm ess}}^\kappa$ to denote
the discrete eigenvalues below the infimum of the essential spectrum of the operators $\mH_\infty$ and
$\mH_\kappa$ respectively.

\begin{theorem}\label{t:strong}
Let $\mH_\kappa=
\int\lambda~\text{d}~E_\kappa(\lambda)$ be the operators which are associated with the
family of forms $\fh_\kappa$.
Take $D\in\R$ such that $\lambda_m^\infty<D<\lambda_{m+1}^\infty$ and set $\fI=\left<-\infty,D\right]$,
then
\begin{align}
\label{e:co0}
\eta_i(\kappa,\fI)&<1,\qquad i=1,\ldots,m,\\
\label{e:co1}
\frac{|\lambda^\kappa_{j}-\lambda_j^\infty|}{\lambda_j^\infty}&\leq\eta_m(\kappa,\fI),\quad j=1,\ldots,m,\\
\label{e:co3}\|E_\kappa(D)-E_\infty(D)\|&\leq
\frac{\sqrt{D\lambda_m^\infty}}{|D-\lambda_m^\infty|}\frac{\eta_m(\kappa,\fI)}{\sqrt{1-\eta_m(\kappa,\fI)}}
\end{align}
for $\kappa$ large enough. (For the meaning of the phrase large enough see Remark \ref{rem:below}.)
\end{theorem}
\begin{proof}
Statement (\ref{e:co0}) is a direct consequence of \cite[Lemma 2.1]{Gru05_3}.
Let us now remember (\ref{eq:rem}). This estimate is the consequence of \cite[Theorem 4.5]{Gru03_3} which,
when applied to the form $\fh_\kappa$ and its $\fE_\infty$ block-diagonal part $(\fh_{\kappa})_{\fE_\infty}$,
yields
$$
(1-\eta_m(\kappa,\fI))\big(\fh_{\kappa})_{\fE_\infty}\leq\fh_\kappa\leq
(1+\eta_m(\kappa,\fI))\big(\fh_{\kappa})_{\fE_\infty}.
$$
Let $\big(\mH_\kappa\big)_{\fE_\infty}$ be the self-adjoint operators which represent the forms
$\big(\fh_{\kappa})_{\fE_\infty}$ in the sense of Kato, then $\lambda_i^\infty\in\spec\big(\mH_\kappa\big)_{\fE_\infty}$. Set
$\mW_\kappa=\big(\mH_\kappa\big)_{\fE_\infty} E_\infty(D)_\perp$
and $\mW_\infty=\mH_\infty E_\infty(D)_\perp$ then Theorem \ref{t:WCon} implies
that $\mW_\infty^\dagger=\slim_{\kappa\to\infty}\mW_\kappa^\dagger$. By the construction
of $\mW_\kappa$ we have $\ra(E_\infty(D))\perp w$ for any $w\in\d(\mW_\kappa)$. This implies
$\lambda_1(\mW_\kappa)\to\lambda_1(\mW_\infty)=\lambda_{m+1}^\infty$.
On the other hand, since $$\spec(\big(\mH_\kappa\big)_{\fE_\infty})=\{\lambda\in\spec(\mH_\infty)~:~\lambda\leq D\}\cup\spec{\mW_\kappa}$$
it follows that there is $\kappa_0$ such that
$$
\big[\lambda_m^\infty,D\big]\subset\R\setminus\spec(\big(\mH_\kappa\big)_{\fE_\infty}),\qquad \kappa>\kappa_0.
$$
Since $\eta_m(\kappa,\fI)\to 0$ we conclude that for $\kappa>\kappa_0$ (here we slightly abuse the notation)
the estimate $\eta_m(\kappa, \fI)\leq\frac{1}{2}\frac{D-\lambda_m^\infty}{D+\lambda_m^\infty}$ holds.
Now, the conclusion (\ref{e:co1}) follows from \cite[Theorem 5.2]{Gru03_3}. Equivalently, the conclusion
(\ref{e:co3}) follows from (\ref{eq:Sylv}) and \cite[Theorem 3.2]{Gru_Ves_Sylv}.
\end{proof}
\begin{remark}\label{rem:below}
The coupling constant $\kappa_0$ is large enough when
$$
\eta_m(\kappa, \fI)<\frac{1}{3}\frac{\lambda_{m+1}^\infty-\lambda_{m}^\infty}{\lambda_{m+1}^\infty+\lambda_m^\infty}
$$
for $\kappa>\kappa_0$. This follows by a similar consideration as in Remark \ref{rem:rel_gap}.
\end{remark}
A direct application of the results from \cite[Section 3]{Gru_Ves_Sylv} and the results of Theorem \ref{t:strong}
is the following corollary.
\begin{corollary}\label{c:working}
Assuming the setting and the notation of the previous theorem we have
$$
\tripleb E_\kappa(D)-E_\infty(D)\tripleb
\leq\frac{\sqrt{D\lambda_m^\infty}}{|D-\lambda_m^\infty|}\frac{\tripleb\diag((\eta_i(\kappa,\fI))_{i=1}^m)
\oplus\diag((\eta_i(\kappa,\fI))_{i=1}^m)\tripleb}
{\sqrt{1-\eta_m(\kappa,\fI)}}.
$$
In the case in which $\fI=\big[D_{-}, D_{+}\big]$ and
$\lambda_{q-1}^\kappa<D_{-}\leq\lambda_{q}^\kappa\leq\lambda_{q+m-1}^{\kappa}\leq D_{+}<\lambda_{q+m}^\kappa$, $\kappa>\kappa_0$
then
\begin{equation}\label{eq:gap}
\| E_\kappa(\fI)-E_\infty(\fI)\|
\leq\Big[\frac{\sqrt{D_{+}\lambda_m^\infty}}{|D_{+}-\lambda_m^\infty|}+
\frac{\sqrt{\lambda_1^\infty D_{-}}}{|\lambda_1^\infty-D_{-}|}
\Big]\frac{\eta_m(\kappa,\fI)}
{\sqrt{1-\eta_m(\kappa,\fI)}}.
\end{equation}
\end{corollary}

An easy comparison with the single operator estimates from Theorem \ref{tm:ess} reveals that,
unlike the spectral family estimate (\ref{e:co3}), the eigenvalue result (\ref{e:co1}) is suboptimal in
the asymptotic setting. The problem is that we can not uniformly apply the estimate (\ref{trece:e_tap6})
on all the operators $\mH_\kappa$, $\kappa>\kappa_0$ since we have no information of the multiplicity
of the eigenvalue $\lambda_i^\kappa$ for all $\kappa>\kappa_0$. We only know the multiplicity of $\lambda_i^\infty$.
The only statement which we can make in general is a lower estimate on the
convergence rate. A way to solve this multiplicity problem will be presented in Section \ref{sec:mult}, for now we only give the following result.
\begin{corollary}\label{c:lower}
Assuming the setting and the notation of Theorem \ref{t:strong} we have
$$
\frac{\lambda^\infty_1}{2\lambda_m^\infty}\sum_{i=1}^m\eta_i^2(\kappa,\fI)
\leq\sum^m_{i=1}\frac{|\lambda_i^\kappa-\lambda^\infty_i|}{\lambda^\infty_i}.
$$
Furthermore, for each $\kappa>0$ we can
chose eigenvectors $v_i^\kappa$, $\mH_\kappa v_i^\kappa=\lambda_i^\kappa v_i^\kappa$, $\|v_i^\kappa\|=1$
and $v_i^\infty$, $\mH_\infty v_i^\infty=\lambda_i^\infty v_i^\infty$, $\|v_i^\infty\|=1$ such that
$$
\frac{\lambda^\infty_1}{2\lambda_m^\infty}\sum_{i=1}^m\eta_i^2(\kappa,\fI)
\leq\sum^m_{i=1}\frac{\fh_\kappa[v_i^\kappa-v^\infty_i]}{\fh\kappa[v^\infty_i]}.
$$
\end{corollary}

One situation in which we can readily obtain upper estimates like those from Theorem \ref{tm:ess} is
the case when we know that $\lambda_i^\infty$ has the multiplicity one. This is frequently a case for the $1D$
differential operators. Also, the
lowest eigenvalue of many Schroedinger operators, like those from \cite{Dancer} have multiplicity one. In what
follows we use $\|\cdot\|_{\mA^{-1}}=\|\mA^{-1/2}\cdot\|$ to denote the standard $\mA^{-1}$-norm, which is
associated to a positive definite operator $\mA$.
\begin{theorem}\label{t:exactness}
Assume the setting and the notation of Theorem \ref{t:strong}, and let $\lambda_q^\infty$, $q\in\N$ be of multiplicity one
then
\begin{align}\label{eq:sharp}
\lim_{\kappa\to\infty}\frac{\frac{\lambda_q^\infty-\lambda_q^\kappa}{\lambda_q^\infty}}{\eta_1^2(\kappa,\lambda_q^\infty)}&=1,\\
\lim_{\kappa\to\infty}\frac{\frac{\fh_\kappa[v_q^\kappa-v_q^\infty]}{\fh_\kappa[v_q^\kappa]}}{\eta_1^2(\kappa,\lambda_q^\infty)}&=1
\label{eq:sharp2}
\end{align}
\end{theorem}
\begin{proof}
By the same argument as above we may assume that we have $\kappa_0$ such that
$$
\eta_1(\kappa, \lambda_q^\infty)\leq\frac{1}{3}\min\{\frac{\lambda_{q+1}^\infty-\lambda_q^\infty}{\lambda_{q+1}^\infty+\lambda_q^\infty},
\frac{\lambda_{q}^\infty-\lambda_{q-1}^\infty}{\lambda_{q}^\infty+\lambda_{q-1}^\infty}\}, \quad\kappa>\kappa_0.
$$
Theorem \ref{t:strong} yields that there exist $D_{-}, D_{+}$ such that $0<D_{-}<\lambda_q^\infty < D_{+}$ and
\begin{align}\label{e:window}
\lambda_{q-1}^\kappa<D_{-}<\lambda^\kappa_q<D_{+}<\lambda_{q+1}^\kappa,\qquad \kappa>\kappa_0.
\end{align}
According to \cite{Gru05_3} we conclude that we may apply the error representation formula (\ref{eq:apply2})
to the operator $\mH_\kappa$ and the test vector $v_q^\infty$, such that
$\mH_\infty v_q^\infty=\lambda_q^\infty v_q^\infty$, $\|v_q^\infty\|=1$. To the vector $v_q^\infty$ we can define
the \textit{residuum} as the functional $\fr_q^\kappa:=\mH_\kappa v_q^\infty-\lambda_q^\infty v_q^\infty$ and
the identity
$$
\|\fr_q^\kappa\|^2_{(\mH_\kappa)_{\fE_\infty}^{-1}}=(v_q^\infty, \mH_\infty v_q^\infty)~\eta^2_1(\kappa,\lambda_q^\infty)
$$ can be
established by an easy computation. Also note the following identities
\begin{align*}
\|\fr_q^\kappa\|_{(\mH_\kappa)_{\fE_\infty}^{-1}}&=\max_{v\in\q\setminus\{0\}}\frac{|\big<\fr_q^\kappa, v\big>|}{\|(\mH_\kappa)_{\fE_\infty}^{1/2}v\|}
= \max_{v\in\q\setminus\{0\}}\frac{|\fh_\kappa(v, v_q^\infty)-(\fh_\kappa)_{\fE_\infty}(v, v_q^\infty)|}
{\|(\mH_\kappa)_{\fE_\infty}^{1/2}v\|}\\&=
\max_{\substack{v\in\q\setminus\{0\}\\
v\perp\je(\fh_e)}}\frac{|\fh_\kappa(v, v_q^\infty)-(\fh_\kappa)_{\fE_\infty}(v, v_q^\infty)|}{\|(\mH_\kappa)_{\fE_\infty}^{1/2}v\|},
\end{align*}
where $\big<\cdot, \cdot\big>$
is the standard duality product.
Analogous manipulation and the error representation formula (\ref{eq:apply2}) yield the conclusion
\begin{align*}
\frac{\frac{\lambda_q^\infty-\lambda_q^\kappa}{\lambda_q^\infty}}{\eta_1^2(\kappa,\lambda_q^\infty)}&=1+
\frac{\lambda_q^\kappa}{\lambda_q^\infty}\frac{( (\mH_\kappa)^{-1}_{\fE_\infty} \fr_q^\kappa,
(\I-\lambda_q^\kappa(\mH_\kappa)_{\fE_\infty}^{-1})^{-1}(\mH_\kappa)^{-1}_{\fE_\infty} \fr_q^\kappa)}
{\eta_1^2(\kappa,\lambda_q^\infty)}\\
&=1+O\big(\|(\mH_\kappa)_{\fE_\infty}^{-1/2} P_{\je(\fh_e)_\perp} \|^2\big).
\end{align*}
Finally, Theorem \ref{t:WCon} implies (\ref{eq:sharp}). The conclusion (\ref{eq:sharp2}) follows from (\ref{eq:strang})
\end{proof}
 We would like
to emphasize that in this result the monotonicity of the family $\fh_\kappa$ played a role. It is possible to
prove the result without the property $\lambda_q^\infty>\lambda_q^\kappa$. The
proof is technically more involved and it does not further the understanding of the problem, so we leave it out.
This theorem establishes that the estimate---which follows directly from Theorem \ref{thm:second}---is sharp.
We formulate this as the following corollary.
\begin{corollary}\label{c:sharp}
Assume the setting of the preceding theorem then
$$
\frac{|\lambda_q^\infty-\lambda_q^\kappa|}{\lambda_q^\infty}\leq\frac{3~\eta_1^2(\kappa,\lambda_q^\infty)}{
\min\{\frac{\lambda_{q+1}^\infty-\lambda_q^\infty}{\lambda_{q+1}^\infty+\lambda_q^\infty},
\frac{\lambda_{q}^\infty-\lambda_{q-1}^\infty}{\lambda_{q}^\infty+\lambda_{q-1}^\infty}\}}
.
$$
This estimate is sharp in the sense of (\ref{eq:sharp}).
\end{corollary}
%
%
%
%
%
%

\subsubsection{A concrete example}\label{section321}
Let $\mH_\kappa$ be the operators which are
defined by the family of positive definite
forms
\begin{equation}\label{e_schrod}
\fh_\kappa(u,v)=\int_0^\infty \partial_x u\partial_xv~dx + \kappa^2\int_1^\infty uv~dx,
\quad u,v\in H^1_0(\R_+).
\end{equation}
Theorem \ref{t:WCon} readily yields
$$
\fh_\infty(u, v)=\int_0^1\partial_x u\partial_xv~dx,\qquad u,v\in H^1_0[0,1].
$$
Here we have used $H^1_0[0,1]$ and $H^1_0(\R_{+})$, $\R_{+}:=\big[0,\infty\big>$ to denote the
standard Sobolev spaces. We also identify the functions from
$H^1_0[0,1]$ with their extension by zero to the whole of $\R_{+}$
and write $ H_0^1[0,1]\subset H_0^1(\R_{+})$.
We also formally write
$\mH_\kappa=-\partial_{xx}+\kappa^2\chi_{\left[1,\infty\right>}$ and
$\mH_\infty=-\partial_{xx}$ and chose
\begin{equation}\label{drugo:e_testfunkcija}
u_i(x)=\begin{cases}\sqrt{2}\sin(k \pi x),&0\leq x\leq 1\\
0,& 1\leq x\end{cases}~, i\in\N
\end{equation}
as a test function(s). A simple computation yields that $\lambda^\kappa_i$ is a solution of the equation
\begin{equation}\label{trece:nonlinear}
\sqrt{\kappa^2-\lambda^\kappa}=-\sqrt{\lambda^\kappa}\cot(\sqrt{\lambda^\kappa})
\end{equation}
and we know that each $\lambda_i^\kappa$ has the multiplicity one.
The quotient
$\frac{\lambda^\infty_1-\lambda_1^\kappa}{\lambda^\infty_1}$ can be
represented (for $\kappa\to\infty$) by a convergent Taylor series (see \cite{VeselicPrivate})
\begin{equation}\label{drugo:e_razvoj}
\frac{\lambda^\infty_1-\lambda_1^\kappa}{\lambda^\infty_1}=
2\frac{1}{\kappa}-3\frac{1}{\kappa^2}+8\left(\frac{1}{2!}+\frac{1}{4!}\pi^2\right)
\frac{1}{\kappa^3}-10\left(\frac{1}{2!}+\frac{4}{4!}\pi^2\right)\frac{1}{\kappa^4}+\cdots~.
\end{equation}
Using the Green functions we also directly compute $\eta^2_1(\kappa,\lambda_i^\infty):=\frac{2}{3+\kappa}$. For
computational details see \cite{GruPhd}.

By utilizing the information from (\ref{trece:nonlinear})
we can establish
\begin{equation}\label{eq:ex-window}
\big(1-\sqrt{\frac{2}{3+\kappa}}\big)4\pi^2=:D(\kappa)\leq\lambda_2(\mH),\qquad
\kappa\geq 5,
\end{equation}
which leads, in combination with (\ref{trece:e_tap6}), to the estimate
\begin{equation}\label{eq:mod1D}
\frac{2}{3+\kappa}\leq\frac{\lambda^\infty_1-\lambda_1^\kappa}{\lambda^\infty_1}\leq\frac{D(\kappa)+\pi^2}{D(\kappa)-\pi^2}\frac{2}{3+\kappa}=
\frac{10}{3\kappa}+\frac{1}{\sqrt{\kappa}}O\Big(\frac{1}{\kappa}\Big),
\qquad \kappa \geq 5.
\end{equation}
Similar sharp results can be obtained for other $\lambda_i^\infty$ and using (\ref{eq:strang})
for corresponding eigenvectors. We tacitly leave out the details.
\subsubsection{A remark on higher dimensional singular obstacle problems}
This paradigm has been applied in \cite{GruPhd} to operators which are defined both in $H^1(\R^n)$ as well as
in $H^1(\Omega)$, where $\Omega\subset\R^n$ is a bounded domain. The only ingredient which is necessary is a result on the behavior
of the momenta
\begin{equation}\label{eq:momenta}
(f, \mH^{-1}_\kappa f)-(f, \mH^\dagger_\infty f)=
\int^\infty_{\kappa^2}\|\mH_e^{1/2}\mH_\tau^{-1}f\|^2~d\tau,\qquad f\in\fE_\infty.
\end{equation}
Estimates of such momenta have been obtained on many places in the literature.
We illustrate our point by a consideration of a model problem of the
electro-magnetic waveguide $\mathcal{O}\times\R$, where the section $\mathcal{O}\subset\R^2$ is a smooth and
connected domain. The material $\Omega\subset\mathcal{O}$ of very large conductivity is compactly immersed in
$\mathcal{O}$, which is to say that the closure $\text{cl}(\Omega)$ is contained in $\mathcal{O}$
and that $\Omega$ is bounded. The dielectric material is now modeled by $\mathcal{U}=\mathcal{O}\setminus\Omega$.
Assuming that the boundary of $\Omega$ is sufficiently smooth we study the
eigenvalue problem for
$
\mH_\kappa=-\triangle+\kappa^2\chi_{\Omega}
$.
Here, $\chi_{\Omega}$ is the characteristic function of $\Omega$ and
$\mH_\kappa$ is the operator which is defined in the sense of Kato by the quadratic
form
$$
\fh_\kappa(\psi, \phi)=\int_{\mathcal{O}}\nabla \psi\cdot\nabla\phi+\kappa^2\int_{\mathcal{O}}\chi_{\Omega}\psi\phi, \qquad
\psi,\phi\in \q_\infty:=H^1_0(\mathcal{O})
$$
where $\kappa\in\R_{+}$ and $\nabla$ is the usual gradient operator on $H^1_0(\mathcal{O})$,
the Sobolev space of functions with zero trace on the boundary $\partial\mathcal{O}$.

This problem has been analyzed in \cite{Gru07_1}. Let us assume that $\lambda_m^\infty<D<\lambda_{m+1}^\infty$, for some $m\in\N$.
We compare $P=E_\infty(\lambda_m^\infty)$ and
$Q_\kappa=E_\kappa(\lambda_m^\infty)$,
where $\mH_\kappa=\int\lambda~dE_\kappa(\lambda)$ is the spectral integral in $L^2(\mathcal{O})$ and
$\mH_\infty=\int\lambda~dE_\infty(\lambda)$
is the spectral integral in $L^2(\mathcal{U})$.
Since, as has been shown in \cite{Gru07_1},
$$
(v_i^\infty, \mH_\kappa^{-1}v_i^\infty)-(v_i^\infty,\mH_\infty^\dagger v_i^\infty)
=\frac{1}{\kappa}\frac{1}{(\lambda_i^\infty)^2}\int_{\partial\Omega}\frac{\partial v_i^\infty}{\partial\nu}
\frac{\partial v_i^\infty}{\partial\nu}+O\big(\frac{1}{\kappa^{3/2}}\big),
$$
we have coarse eigenvector estimates
\begin{align}\label{eq:to_be}
\|Q_\kappa-P\|
&\leq\frac{\sqrt{D\lambda^\infty_m}}{
D-\lambda^\infty_m}\frac{1}{\sqrt{\kappa}}+
O\big(\frac{1}{\kappa^{3/4}}\big)\leq\frac{4}{\frac{\lambda_{m+1}^\infty-\lambda_m^\infty}
{\lambda_{m+1}^\infty+\lambda_m^\infty}}\frac{1}{\sqrt{\kappa}},\\
\label{eq:to_be_2}
\min_{i=1,\cdots,m}\frac{\int_{\partial\Omega}
\frac{\partial v_i^\infty}{\partial\nu}
\frac{\partial v_i^\infty}{\partial\nu}}{\lambda_i^\infty}\frac{1}{2\kappa}&\leq\frac{\fh_\kappa[v_i^\kappa-v_i^\infty]}{\fh_\kappa[v_i^\infty]}\leq
\frac{4}{\frac{\lambda_{m+1}^\infty-\lambda_m^\infty}
{\lambda_{m+1}^\infty+\lambda_m^\infty}}\frac{1}{\kappa}
\end{align}
which can be improved in a straight forward manner by bringing the factor $\int_{\partial\Omega}\frac{\partial v_i^\infty}{\partial\nu}
\frac{\partial v_i^\infty}{\partial\nu}$ into estimates, as has been shown in \cite[Section 2.1]{Gru07_1}. The last inequality in (\ref{eq:to_be})
and (\ref{eq:to_be_2}) hold for $\kappa$ large enough.
The optimal eigenvalue estimate can easily be constructed from Theorem \ref{tm:ess} and \ref{t:exactness} and
we know that the eigenvector estimate (\ref{eq:to_be_2}) is optimal in the sense of
(\ref{eq:strang}) and (\ref{eq:sharp2}). Remark \ref{rem:below}
indicates how to assess the radius of convergence of these first order estimate(s).

\subsubsection{Remarks on
(finite) eigenvalues in gaps of essential spectrum and on general converging families $\fh_\kappa$}
We have said that the theory can be applied to eigenvalues which are in the gaps
of essential spectrum. Since we do not consider any model examples which show such behavior (e.g. operators with
periodic boundary conditions) we will only briefly outline a possibility to obtain results like
Theorem \ref{t:exactness} or Theorem \ref{t:strong} in this setting.

In dealing with the eigenvalues in gaps of the essential spectrum we do not have the safe
convergence environment of Theorem \ref{t:WCon}. Instead, we have to have an \textit{a priori}
information that the assumption like (\ref{e:window}) holds. An example of how to obtain this
type of \textit{a priori} information can be seen on the proof of (\ref{eq:ex-window}). To this end we
would like to emphasize that such type of ``precise'' result on the separation of the target eigenvalue
from the unwanted component of the spectrum is an unavoidable ingredient of all constructive spectral estimates.
An assumption like (\ref{e:window}) is equivalent to requiring that eigenvalue $\lambda_q$ be stable under the
perturbation $\fh_\kappa$, see \cite[chapter VIII.4, pp. 437]{Kato76}. For a characterization
of perturbations for which this holds see \cite{Linden1} and references therein.

Given such an estimate---i.e. assuming that $\lambda_q$ is a stable eigenvalue---the appropriate
result from \cite{Gru05_3} or \cite{Gru_Ves_Sylv} can be applied to
obtain convergence estimates. We also emphasize that the theory of \cite{Gru05_3,Gru_Ves_Sylv}
allows for more general spectral intervals $\fI$. To be more precise, to establish an estimate like (\ref{eq:gap})
the spectrum in $\fI\cap\spec(\mH_\infty)$ does not have to be discrete. However, in such
situation we have no guarantee that  $\eta_{\max}(\kappa,\fI)<1$ and obtaining computational formulae
requires much more technical work. The precise use in a given situation is application dependent, but
always follows the procedure outlined in Theorems \ref{t:strong} and \ref{t:exactness}.

In the case in which we consider a general converging family of quadratic forms from \cite{WeidmannScand84} we cannot
conclude that $(\mH^{\dagger}_\infty)_{E_\infty(\fI)}=(\mH_\kappa)^{-1}_{E_\infty(\fI)}$, so we have to use
explicitly computable $\Xi_\kappa^{-1}:=(\mH_\kappa)^{-1}_{E_\infty(\fI)}\big|_{E_\infty(\fI)}$ in (\ref{eq:AppDef})
instead. If we set $\mu_{i}^\kappa:=\lambda_i(\Xi_\kappa)$, then $\mu_i^\kappa$ substitutes for
$\lambda_i^\infty$ in eigenvalue estimates like (\ref{e:co1}), (\ref{eq:sharp}), whereas the
estimates for the spectral projections like (\ref{eq:gap}) remain unchanged, e.g. we have the convergence estimate
$
\|Q_\kappa-P\|\leq\frac{\sqrt{\lambda_{m+1}^\kappa\mu_m^{\kappa}}}{|\lambda_{m+1}^\kappa-\mu_m^{\kappa}|}
\frac{\eta_m(\kappa,P)}{\sqrt{1-\eta_m(\kappa,P)}}
$.

\subsubsection{A method to solve the multiplicity problem}\label{sec:mult}
A tacit assumption in this semiclassical analysis is that the operator $\mH_\infty$ is a well known object.
In order to be able to apply Theorem \ref{tm:ess}
one should establish that there exists $\kappa_0>0$ such that
\begin{equation}\label{drugo:e_vvH}
\lambda^\kappa_{q-1}<
D_{-}<\lambda^\kappa_q=\lambda^\kappa_{q+m-1}< D_{+}<
\lambda^\kappa_{q+m},
\end{equation}
for $\kappa>\kappa_0$.
However, if $m>1$ it is not plausible to expect that (\ref{drugo:e_vvH})
will hold in general. Instead, we will get a tight
cluster of $m$ eigenvalues (counting the eigenvalues
according to their multiplicity) that converge to
$\lambda_q^\infty$. Since we aim to express the spectral information about $\mH_\kappa$ in terms
of the spectrum of $\mH_\infty$ we further opt to give specific values for $D_{-}$ and $D_{+}$ as
functions of the gaps in the spectrum of $\mH_\infty$.

\begin{theorem}\label{t:Multi}
Let the eigenvalues of the operator $\mH_\infty$ be so ordered
that
$
\lambda^\infty_{q-1}<\lambda^\infty_q=\lambda^\infty_{q+m-1}<
\lambda^\infty_{q+m}
$.
Define the measure of the
relative separation of $\lambda_q^\infty$ from the rest of
the spectrum of $\mH_\infty$ as
the number
$$
\gamma_s(\lambda_q^\infty)=
\min\left\{\frac{\lambda^\infty_{q+m}-\lambda_q^\infty}{\lambda^\infty_{q+m}+\lambda_q^\infty},
\frac{\lambda^\infty_{q}-\lambda_{q-1}^\infty}{\lambda^\infty_{q}+\lambda_{q-1}^\infty}\right\}.
$$
There exists $\kappa_0>0$ such that for  $\kappa\geq\kappa_0$
\begin{equation}\label{e_pazi_ sad_3}
\frac{|\lambda^\kappa_{q+i-1}-\lambda^\infty_q|}{\lambda^\infty_q}
<\eta_m(\kappa,\lambda_q^\infty)
\frac{\frac{3\eta_m(\kappa,\lambda_q^\infty)}{\gamma_c(\lambda_m^\infty)}}{1-\frac{3\eta_m(\kappa,\lambda_q^\infty)}
{\gamma_s(\lambda_m^\infty)}},
\quad i=1, \ldots,m.
\end{equation}
\end{theorem}
\begin{proof}
Since $\eta_m(\kappa,\lambda_q^\infty)\to 0$,
an argument analogous to the argument that led to Theorem \ref{t:strong} implies that
we can pick $\kappa_0>0$ such that for $\kappa>\kappa_0$
\begin{align}\label{cetvrto:assumption}
\eta_m(\kappa,\lambda_q^\infty)&\leq\frac{1}{3}\gamma_s(\lambda^\infty_q)\\
|\lambda^\kappa_k-\lambda^\infty_q|&
\leq\frac{1}{3}\gamma_s(\lambda^\infty_q)\lambda^\infty_q,\qquad k=q, q+1, \ldots,q+m-1,
\label{drugo:e_splitting}
\\
|\zeta-\lambda_k(\widehat{\mH}_\kappa)|&>\frac{1}{3}\gamma_s(\lambda^\infty_q)\lambda_k(\widehat{\mH}_\kappa)
,\;\; k\not\in\{q, \ldots,q+m-1\},\;\;\zeta\in\fC(\lambda_q^\infty).
\label{drugo:e_splitting2}
\end{align}
Here $\fC(\lambda_q^\infty)$
is the circle in the complex plane with the radius
$\frac{1}{3}\gamma_s(\lambda^\infty_q)\lambda^\infty_q$
and the center $\lambda^\infty_q$.
Assume $\kappa>\kappa_0$ is fixed, then define the family
\begin{equation}\label{drugo:drmacevafamilija}
\fa(\tau)=(\fh_\kappa)_P +\tau\delta_P(\fh_\kappa),\qquad \tau\in\C.
\end{equation}
This is a \textit{holomorphic family of type (B)}
\index{holomorphic family of type (B)} (for the definition see \cite[Chapter VII]{Kato76}).
We know that
\begin{equation}\label{trece:wellscaled}
|\delta_P(\fh_\kappa)[u]|<\eta_m(\kappa,\lambda_q^\infty) (\fh_\kappa)_P[u],\qquad u\in\q,
\end{equation}
so \cite[Theorem VII-4.9 and (VII-4.45)]{Kato76} imply that the resolvent
$$
R(\tau,\zeta)=(\mA(\tau)-\zeta\I)^{-1}
$$
can be represented by a convergent power series in $\tau$ for $\zeta\in\fC(\lambda_q^\infty)$.
The power series for $R(\tau, \zeta)$ converges for every
\begin{equation}\label{drugo:radiuskonv}
|\tau|<r_0=\frac{1}{\eta_m(\kappa,\lambda_q^\infty)}
\inf_{\substack{\zeta\in\fC(\lambda_q^\infty),\\ \lambda\in\spec((\mH_\kappa)_P)}}\frac{|\lambda-\zeta|}{\lambda}=
\frac{1}{\eta_m(\kappa,\lambda_q^\infty)}\frac{1}{3}\gamma_s(\lambda^\infty_q).
\end{equation}
In particular, assumption (\ref{cetvrto:assumption}) implies that the series converges for $\tau=1$.

Define
$$
\widehat{B}(\tau):=-\frac{1}{2\pi i}
\mA(\tau)\int_{\fC(\lambda_q^\infty)} R(\tau,\zeta)~d\zeta,
$$
then $\widehat{B}(\tau)$ is a holomorphic operator family and
there exist $m$ holomorfic functions $\widehat{\lambda}_i(\tau)$ such that
$\widehat{\lambda}_1(\tau), \cdots, \widehat{\lambda}_m(\tau)$ are all
the nonzero eigenvalues of the operator $\widehat{B}(\tau)$.
Due to the assumptions we have made it follows that for $i=1, \ldots,m$
$$
|\widehat{\lambda}_i(\tau)-\lambda^\infty_q|<\frac{1}{3}\gamma_s(\lambda^\infty_q)\lambda^\infty_q,
\qquad|\tau|<r_0.
$$
Cauchy's integral inequality\footnote{For further details see
\cite[Section 8.1.4]{BaumAnalitic} and \cite[Section II-3]{Kato76}.}
for the coefficients of the Taylor expansion implies, for every $i=1, \ldots, m$,
the estimate
$$
|\widehat{\lambda}_i^{(n)}|<
\frac{\frac{1}{3}\gamma_s(\lambda^\infty_q)\lambda^\infty_q}{r_0^n},\qquad n=1,2,\cdots
$$
where
$
\widehat{\lambda}_{i}(\tau)=\lambda^\infty_q+\tau\widehat{\lambda}_{i}^{(1)}
+\tau^2\widehat{\lambda}_{i}^{(2)}+\tau^3\widehat{\lambda}_{i}^{(3)}+\cdots
$.
This yields
$$
|\widehat{\lambda}_i(\tau)-\lambda^\infty_q-\tau\widehat{\lambda}_{i}^{(1)}|<
\frac{\frac{1}{3}\gamma_s(\lambda^\infty_q)\lambda^\infty_q}{r_0}\frac{|\tau|^2}{r_0-|\tau|}
\leq\frac{\frac{1}{3}\gamma_s(\lambda^\infty_q)\lambda^\infty_q}{r_0^2}\frac{|\tau|^2}{1-\frac{|\tau|}{r_0}}
$$
for $|\tau|<r_0$. In particular for $\tau=1$ there exists a permutation
$j_{(\cdot)}$ such that
$\widehat{\lambda}_{j_i}(1)=\lambda^\kappa_{q+i-1}$, $i=1, \ldots,m$ so
$$
|\lambda^\kappa_{q+i-1}-\lambda^\infty_q-\widehat{\lambda}_{j_i}^{(1)}|
<\eta_m(\kappa,\lambda_q^\infty) \lambda^\infty_q ~\frac{3\eta_m(\kappa,\lambda_q^\infty)}{\gamma_c(\lambda^\infty_q)}~
\frac{1}{1-\frac{3\eta_m(\kappa,\lambda_q^\infty)}{\gamma_s(\lambda^\infty_q)}}.
$$
With this is the proof of the theorem finished. To see this note that it was established, in \cite[(VII-4.50)]{Kato76}, that
$
\widehat{\lambda}_{j_i}^{(1)}
$
are the eigenvalues of the matrix $M_{kp}=\delta_P(\fh_\kappa)(u_k,u_p)$,
where $u_k$, $k=1, \ldots, m$ form an orthonormal basis for $\ra(E_\infty[D_{-}D_{+}])$. Since
$$\delta_P(\fh_\kappa)(u, v)=\fh_\kappa(P_\perp u,Pu)+\fh_\kappa(Pu,P_\perp u)=0,\qquad u,v\in \ra(P),
$$
we obtain $\widehat{\lambda}_{j_i}^{(1)}=0$, $i=1, \ldots, m$ and the conclusion follows.
\end{proof}
\begin{remark}
The estimate of this theorem is optimal in the sense of Corollary \ref{c:lower}. The upper estimate
which has a similar form to (\ref{trece:e_tap6}) can be established for the limit eigenvalues
$\lambda_1^\infty\leq\cdots\leq\lambda_m^\infty$.
The the role of the constant from (\ref{trece:e_tap6}) is taken by the
constant $\gamma_{\min}(\lambda_m^\infty):=\min\{\gamma_c(\lambda_i^\infty)~:~i=1, \ldots, m\}$,
as is given by the repeated application of Theorem \ref{t:Multi}.
We leave out the technical details.
\end{remark}
\section{Spectral asymptotic in the regular case}\label{section4}

We now concentrate on the non-inhibited families
\begin{equation}\label{eq:defNI}
\fh_\kappa(u, v)=\fh_b(u, v) +\kappa^2 \fh_e(u, v), \qquad u, v\in\q:=\q(\fh_b)\subset\q(\fh_e),
\end{equation}
which satisfy the additional regularity assumption
that the range of the operator $\mH_e^{1/2}\mH_b^{-1/2}$ is closed in $\H$. As already mentioned
in Section \ref{sec:reg} this is equivalent with
\begin{equation}\label{eq:BB2}
\|(\mH_e^{1/2}\mH_b^{-1/2})^\dagger\|=\fk<\infty.
\end{equation}
With this additional requirement, which has a flavor of Linear Algebra, we can use an adaptation
of the Lagrange-Multiplier
technique to establish an upper estimate for the momenta
\begin{equation}\label{eq:JT}
(f, \mH^{-1}_\kappa f)-(f, \mH_\infty^\dagger f)=\int^\infty_{\kappa^2}\|\mH_e^{1/2}\mH_\tau^{-1}f\|^2~d\tau,
\qquad f\in\q_\infty:=\q(\fh_\infty).
\end{equation}
The lower estimate for (\ref{eq:JT}) follows by an adaptation of the spectral-calculus technique
from \cite{Brasche,Brasche1}. With this we prove the optimality of our approach to spectral asymptotic estimation.

The following lemmata are the main technical results which are needed to estimate the quantities (\ref{eq:JT}).

\begin{lemma}\label{l:En}
Take $f\in\overline{\q_\infty}^{~_{\|\cdot\|}}$, then
$
\fh_\kappa[\mH^{-1}_\kappa f-\mH^\dagger_\infty f]=(f,\mH^{-1}_\kappa f)-(f, \mH^\dagger_\infty f)
$.
\end{lemma}
\begin{proof}
The proof is a straight forward computation. Take $f\in\overline{\q_\infty}^{~_{\|\cdot\|}}$, then
$$
h_\kappa[\mH^\dagger_\infty f]=(f,\mH^\dagger_\infty f)
$$
and we have
\begin{align*}
h_\kappa[\mH^{-1}_\kappa f-\mH^\dagger_\infty f]&=(f,\mH^{-1}_\kappa f)-
h_\kappa(\mH^{-1}_\kappa f,\mH^\dagger_\infty f)-h_\kappa(\mH^\dagger_\infty f,\mH^{-1}_\kappa f)
+(f,\mH^\dagger_\infty f)\\
&=(f,\mH^{-1}_\kappa f)-(\mH^{-1/2}_\kappa f,\mH^{1/2}_\kappa\mH^\dagger_\infty f)
-(\mH_\kappa^{1/2}\mH^\dagger_\infty f,\mH^{-1/2}_\kappa f)\\
&\qquad + (f,\mH^\dagger_\infty f)\\
&=(f,\mH^{-1}_\kappa f)-(f,\mH^\dagger_\infty f).
\end{align*}
\end{proof}
\begin{lemma}\label{l:TamL}
Let $f\in\H$ be given then set $r_f:=\mH^{-1/2}_bf-\mH_b^{1/2}\mH_\infty^\dagger f$. If we assume
$\|(\mH_e^{1/2}\mH_b^{-1/2})^\dagger\|<\infty$
then $q_f=(\mH_e^{1/2}\mH_b^{-1/2})^\dagger r_f$ and
$$
\fh_b(\mH^\dagger_\infty f, v)+ (q_f, \mH_e^{1/2} v)= (f, v), \qquad v\in\q.
$$
Furthermore, it holds that $\|r_f\|^2=(f, \mH^{-1}_b f)-(f, \mH_\infty^\dagger f)$.
\end{lemma}
\begin{proof}
It holds that $r_f\perp\mH_b^{1/2}\q_\infty$, which can be checked by a direct computation. The operator
$\mB:=(\mH_e^{1/2}\mH_b^{-1/2})$ has the closed range so
$$
\H=\ra(\mB^*)\oplus\je(\mB)=
\ra(\mB^*)\oplus\mH_b^{1/2}\q_\infty.
$$
Therefore we have $r_f\in\ra(\mB^*)$ and so we may write $q_f:=\mB^\dagger r_f$. A direct computation
now shows that
\begin{align*}
\fh_b(\mH^\dagger_\infty f, v)+ (q_f, \mH_e^{1/2} v)&=(\mH_b^{1/2}\mH_\infty^\dagger f, \mH_b^{1/2}v)
+(\mB^{*\dagger}r_f, \mB\mH_b^{1/2} v)\\
&=(\mH_b^{1/2}\mH_\infty^\dagger f + r_f, \mH_b^{1/2}v)=(f, v).
\end{align*}
\end{proof}
The main quantitative theorem about the asymptotic behavior of (\ref{eq:defNI}) follows now directly.
\begin{theorem}\label{t:GT}
Assume $\fk:=\|(\mH_e^{1/2}\mH_b^{-1/2})^\dagger\|<\infty$ then we have; for $f\in\overline{\q_\infty}^{~_{\|\cdot\|}}$;
\begin{equation}\label{eq:uTG}
\frac{{(f, \mH_1^{-1}f)\!-\!(f, \mH_\infty^\dagger f)}}{\kappa^2}\leq{(f, \mH^{-1}_\kappa f)-(f, \mH_\infty^\dagger f)}\leq
\frac{\fk^2\big({(f, \mH_b^{-1}f)\!\!-\!\!(f, \mH_\infty^\dagger f)}\big)}{\kappa^2}
\end{equation}
and
\begin{equation}\label{eq:lTG}
\frac{1}{\kappa^2}\eta_i^2(1,\lambda^\infty)\leq\eta_i^2(\kappa,\lambda^\infty)\leq\frac{\fk^2}{\kappa^2}\eta_i^2(0,\lambda^\infty),
\qquad i=1,\ldots,m~,
\end{equation}
where $m$ is the multiplicity of the discrete eigenvalue $\lambda^\infty$
(not necessarily below the infimum of the essential spectrum of $\mH_\infty$).
\end{theorem}
\begin{proof}
For any $f\in\H$ we have
\begin{align*}
\fh_b(\mH^\dagger_\infty f,v)+(q_f,\mH^{1/2}_e v)&=(f,v),\qquad v\in\q,\\
\fh_b(\mH^{-1}_\kappa f,v)+\kappa^2\fh_e(\mH^{-1}_\kappa f,v)&=(f,v),\qquad v\in\q.
\end{align*}
which implies
$$
\fh_b(\mH^{-1}_\kappa f-\mH^\dagger_\infty f, v)+\kappa^2\fh_e(\mH_\kappa^{-1} f, v)=(q_f, \mH_e^{1/2}v)
$$
and subsequently
$$
\kappa^2 \fh_e[\mH_\kappa^{-1}f]\leq\|q_f\|\fh_e[\mH_\kappa^{-1}f]^{1/2}.
$$
The right inequality in (\ref{eq:uTG}) follows from Lemma \ref{l:TamL}. To establish the left inequality
of (\ref{eq:uTG}) we start
from the identity \cite[(22)]{Brasche}. We combine the integral representation for $(f,\mH^{-1}_\kappa
f)-(f, \mH^\dagger_\infty f)$ from \cite[pp. 41]{Brasche} and \cite[(29)]{Brasche} to obtain
\begin{align*}
(f,\mH^{-1}_\kappa f)-(f, \mH^\dagger_\infty f)&=\int^\infty_0
\frac{1}{\lambda+\kappa^2\lambda^2}
(d E_{\mH_e}(\lambda)\mH_e^{1/2}\mH^{-1}_b f, \mH_e^{1/2}\mH^{-1}_b f)\\
&=((\I+\kappa^2\mH_e)^{-1}\mH_b^{-1} f, \mH_b^{-1} f)\\
&\geq\frac{1}{\kappa^2}((\I+\mH_e)^{-1}\mH_b^{-1} f, \mH_b^{-1} f)\\
&=\frac{1}{\kappa^2}\big((f,\mH^{-1}_1 f)-(f, \mH^\dagger_\infty f)\big).
\end{align*}
The conclusion (\ref{eq:lTG}) for the approximation defects follows directly from
the definition (\ref{eq:AppDef}) and the observaton that
$$
\frac{1}{(f,\mH^{-1}_1 f)}\leq\frac{1}{(f, \mH^{-1}_\kappa f)}\leq\frac{1}{(f, \mH_\infty^\dagger f)},\qquad
f\in\ra\big(E_\infty(\{\lambda^q\})\big),
$$
holds. This completes the argument.
\end{proof}
 \begin{example}\label{vazni_primjer}
 We will present this example as an abstract variation on (\ref{e:heat}).
 Let $\mH$ be a positive definite operator, let $P$ be a projection, $\ra(P)\subset\d(\mH^{1/2})$
 and let $r^\kappa_f:=\mH_\kappa^{-1/2}f-\mH_\kappa^{1/2}\mH_\infty^\dagger f$. Consider
 $$
 \fh_\kappa(u, v)=((\I+\kappa^2P)\mH^{1/2} u, \mH^{1/2} v)=\fh_b(u,v)+\kappa^2\fh_e(u,v),
 $$
 then
 $$
 \|(\mH^{1/2}_e\mH^{-1/2}_b)^\dagger\|\leq 1
 $$
 and (\ref{eq:uTG}) gives for $f\in\je(P\mH_e^{1/2})$
\begin{equation}\label{drugo:s1}
\frac{\|r_f^\kappa\|^2}{\|r_f\|^2}=
\frac{ (f, \mH^{-1}_\kappa f)-(f, \mH^{-1/2}P_\perp\mH^{-1/2} f)}
{(\mH^{-1/2}f, P\mH^{-1/2} f)}\leq\frac{1}{\kappa^2}.
\end{equation}
Here we have used $\|r_f\|^2=(f, \mH_b^{-1}f) - (f, \mH_\infty^\dagger f)$
and $\|r_f^\kappa\|^2=(f, \mH_\kappa^{-1}f) - (f, \mH_\infty^\dagger f)$ to simplify the notation.
On the other hand, we compute
 $$
 \mH_\kappa^{-1}=\mH^{-1/2}(P_\perp +\frac{1}{1+\kappa^2}P)\mH^{-1/2}
 $$
to establish
\begin{equation}\label{drugo:s2}
(f, \mH_\kappa^{-1}f)-(f, \mH^{-1/2}P_\perp \mH^{-1/2}f) =\frac{1}{1+\kappa^2}(\mH^{-1/2}f, P\mH^{-1/2}f).
\end{equation}
Formulae (\ref{drugo:s1}) and (\ref{drugo:s2}) give
$$
\frac{1}{1+\kappa^2}=\frac{\|r_f^\kappa\|^2}{\|r_f\|^2}\leq\frac{1}{\kappa^2},
$$
which is a very favorable estimate for $\kappa$ large.
The lower estimate can be computed to be
$$
\frac{1}{2\kappa^2}\leq\frac{\|r_f^\kappa\|^2}{\|r_f\|^2},
$$
which is not as sharp as the upper estimate, but it is---newer the less---asymptotically optimal.
 \end{example}

\section{A model problem from 1D theory of elasticity}\label{s:Arch}
 As an illustration of the applicability of Theorem \ref{t:GT}, we consider the small frequency problem
for the circular arch as described in \cite[Chapter 8.8:3]{CiarletV4y78} and \cite{SanchezPalencia90},
cf. Figure \ref{drugo:f_curvedrod}.
Let $\mathbf{\phi}:[0,l]\to\R^2$ be the middle curve of the arch. We take $\mathbf{\phi}$
to be the upper part of the circle with the radius $R$.
The arch (the model problem we are considering) will be a thin homogeneous, elastic body
of the constant cross-section $\mathcal{A}$, whose area is $A>0$. The arch will be clamped
at one end and free at the other. The strain energy of the arch is given\footnote{See also
\cite{TamTeza}.}
by the positive definite form
\begin{align}
\fa(\mathbf{u},\mathbf{v})&=EI\int_{0}^l\left(u_2'+\frac{u_1}{R}\right)'
\left(v_2'+\frac{v_1}{R}\right)'~d s+
EA\int_{0}^l\left(u_1'-\frac{u_2}{R}\right)\left(v_1'-\frac{v_2}{R}\right)~d s,\label{e_energijaluka}\\
\nonumber&\mathbf{u},\mathbf{v}\in
\q(a)=\{\mathbf{u}\in H^1[0,l]\times H^2[0,l]~:~\mathbf{v}(0)=0,v'_2(0)=0\}.
\end{align}
Here $\mathbf{u}=(u_1,u_2)$ and $\mathbf{v}=(v_1, v_2)$ are the functions of the curvilinear abscissa
$s\in[0,l]$´, the constant $E$ is the Young modulus of elasticity, the constant $A$ is
the area of the cross-section $\mathcal{A}$
and the constant $I$ is the moment of inertia of the cross-section $\mathcal{A}$
.

Let us assume we have the referent arch with the cross-section area
$A$ and the cross-section moment $I$. We consider the family of rods
whose cross-section and the moment of inertia of the cross-section behave like
$$
A_\kappa=\frac{1}{\kappa^2}A=\varepsilon^2A,\qquad I_\kappa=\frac{1}{\kappa^4}I=\varepsilon^4 I.
$$
We want to study the spectral properties of this family of arches as $\varepsilon\to 0$.
More general arch models can be treated by analogous procedures. This is a subject for future reports.
\begin{figure}[t]
\begin{center}\includegraphics[width=9cm]{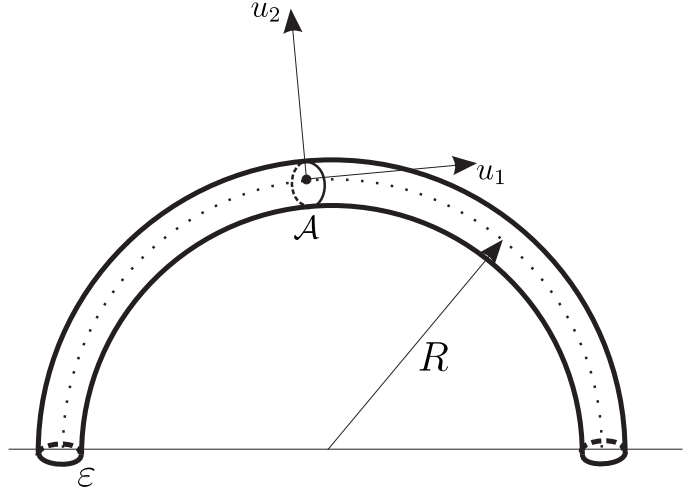}\end{center}
\caption{The Curved rod model}\label{drugo:f_curvedrod}
\end{figure}

For some given $\kappa>0$, $\kappa:=\varepsilon^{-1}$, we write
\begin{align*}
\fa_\kappa(\mathbf{u},\mathbf{v})
&=\frac{E~ I}{\kappa^4}\int_{0}^l\!\!\left(u_2'+\frac{u_1}{R}\right)'\!
\left(v_2'+\frac{v_1}{R}\right)'\text{d}s
+\frac{E~A}{\kappa^2}\int_{0}^l\!\!\left(u_1'-\frac{u_2}{R}\right)\!\left(v_1'-\frac{v_2}{R}\right)d s
\end{align*}
and use $\mA_\kappa$ to denote the operator which is defined by $\fa_\kappa$.
Since $\mathbf{A}_\kappa$ has only the discrete spectrum we write $\lambda_i(\mathbf{A}_\kappa)$,
$i\in\N$. After rescaling
$$
\lambda_i(\mathbf{A}_\kappa)=\frac{1}{\kappa^4}\lambda_i^\kappa
$$
we see that $\lambda^\kappa_i$ are the eigenvalues of the operator $\mH_\kappa$, which
is defined by
\begin{align*}
\fh_\kappa(\mathbf{u},\mathbf{v})&=\fh_b(\mathbf{u},\mathbf{v})
+\kappa^2\fh_e(\mathbf{u},\mathbf{v})\\
&=EI\int_{0}^l\left(u_2'+\frac{u_1}{R}\right)'
\left(v_2'+\frac{v_1}{R}\right)'~d s+
\kappa^2EA\int_{0}^l\left(u_1'-\frac{u_2}{R}\right)\left(v_1'-\frac{v_2}{R}\right)~d s
\end{align*}
for $\mathbf{u},\mathbf{v}\in\q(\fa_\kappa)=\q(\fh_\kappa)$.
Since $\lambda_i^\kappa$ enable us to describe only the eigenvalues of $\mathbf{A}_\kappa$ for
which
$$
\lim_{\kappa\to \infty}\frac{1}{\kappa^4}\lambda_i(\mathbf{A}_\kappa)<\infty.
$$
here we see where the name ``\textit{low frequency problem}"\index{frequency!low frequency problem},
for the eigenvalue problem for $\mH_\kappa$, comes from. The low frequency problem
satisfies the conditions of Theorem \ref{t:WCon}, so we conclude
that the limiting form is
\begin{equation}\label{e_tambacastap}
\fh_\infty(\mathbf{u},\mathbf{v})=EI\int_{0}^l\left(u_2'+\frac{u_1}{R}\right)'
\left(v_2'+\frac{v_1}{R}\right)'~\text{d}s,\quad
\mathbf{u},\mathbf{v}\in\{\mathbf{f}\in\q(a_\kappa),f_1'-\frac{f_2}{R}=0\}.
\end{equation}
In \cite{TamTeza} it has shown that (\ref{e_tambacastap}) is the strain energy of the Curved Rod
Model and that $\fh_\kappa$, $\kappa >0$ are positive definite with $$
\q(h_\kappa)=\{\mathbf{u}\in H^1[0,l]\times H^2[0,l]~:~\mathbf{v}(0)=0,v'_2(0)=0\}.
$$

\begin{remark}\label{trece:curvedrodmodel}
From (\ref{e_tambacastap}) we can see the significance of the condition
\begin{equation}\label{trece:inextensibility}
f_1'-\frac{f_2}{R}=0.
\end{equation}
Assume the rod is locally straight. That is to say, assume $R\to\infty$, then (\ref{trece:inextensibility})
turns into
$$
f_1'=0,
$$
a condition of the inextensibility of the middle curve of the straight rod. The fact that
$f_1'-\frac{f_2}{R}=0$ is an \textit{inextensibility condition}\index{inextensibility condition}
for the middle curve of the curved rod can be established by a rigorous differential
geometric argument, see \cite{TamTeza}. Continuing this heuristic reasoning, we conclude
that Curved Rod model describes the transversal vibrations (perpendicular
to the middle curve) of the curved rod. Arch Model ``couples"
the longitudinal vibrations of the rod with the transversal vibrations.
The study of finer properties of longitudinal vibrations requires the analysis of the so
called ``\textit{middle frequency problem}", which will not be further considered here. However, since the
``\textit{middle frequency problem}" also falls under the scope of Theorem \ref{t:WCon} this theory could also be
applied in that case, too.
\end{remark}

\subsection{Computational details}
Based on (\ref{e_energijaluka}) and (\ref{e_tambacastap}) one concludes that
the sequence $\fh_\kappa$ satisfies the assumptions of Theorem \ref{t:GT}.
Here is a word of additional explanation in order. We have formulated all of our results
about the forms $\fh_b$ and $\fh_e$ based
on the representations
\begin{align*}
\fh_b(u,v)&=(\mH^{1/2}_b u,\mH^{1/2}_b v),\\
\fh_e(u,v)&=(\mH^{1/2}_e u,\mH^{1/2}_e v).
\end{align*}
However, we can represent (see (\ref{e_tambacastap})) the forms
$\fh_b$ and $\fh_e$ with the help of the operators $\mR_b:\q(\fh_b)\to\H_b$ and $\mR_e:\q(\fh_e)\to\H_e$.
The only assumptions on the operators $\mR_b$ (and $\mR_e$) is that they have a closed range in the
auxiliary Hilbert spaces $\H_b$ (and $\H_e$), cf. \cite{GruVes02}. The representation
theorem for the nonnegative definite forms implies
\begin{align}
\fh_b(u,v)&=(\mH^{1/2}_b u,\mH^{1/2}_b v)=(\mR_b u,\mR_b v)_{\H_b}\label{e_cholesky1},\\
\fh_e(u,v)&=(\mH^{1/2}_e u,\mH^{1/2}_e v)=(\mR_e u,\mR_e v)_{\H_e},\label{e_cholesky2}
\end{align}
where $(\cdot, \cdot)_\x$ generically denotes the scalar product in the Hilbert space $\x$.
The relations (\ref{e_cholesky1}) and (\ref{e_cholesky2}) imply that there exist isometric isomorphisms
$Q_b:\H_b\to\H$ and $Q_e:\H_e\to\H$ such that
$
\mH^{1/2}_b=Q_b\mR_b$, $\mH^{1/2}_e=Q_e\mR_e$,
and in particular
\begin{align*}
(\mH^{1/2}_b u,\mH^{1/2}_b v)&=(Q_b\mR_b u,Q_b\mR_b v)=(\mR_b u,\mR_b v)_{\H_b},\\
(\mH^{1/2}_e u,\mH^{1/2}_e v)&=(Q_e\mR_e u,Q_e\mR_e v)=(\mR_e u,\mR_e v)_{\H_e}.
\end{align*}
We also have for $\mathbf{u}\in\q(\fh_b)$
\begin{align*}
Q^{-1}_b\mH^{1/2}_b\mathbf{u}&=\mR_b\mathbf{u}=\sqrt{E~I}\left(u_2'+\frac{u_1}{R}\right)',\\
Q^{-1}_e\mH^{1/2}_e\mathbf{u}&=\mR_e\mathbf{u}=\sqrt{E~A}\left(u_1'-\frac{u_2}{R}\right)
\end{align*}
and $\mR_b:\q(\fh_b)\to\H_b=L^2[0, l]$ and $\mR_e:\q(\fh_e)\to\H_e=L^2[0, l]$.

Note that $\mH_b$ is not positive definite but $\mH_1$, which is defined by the form
$\fh_1=\fh_b+\fh_e$, is. For the details see \cite{JurakTambaca,Tambaca1D}.
If we were to change the notation we would have to set $\widetilde{\fh}_b:=\fh_1$.
Since this would unnecessarily complicate the exposition we opt not to do so.

We show that
\begin{equation}\label{trece:radiusocjena}
\|(\mH^{1/2}_e\mH^{-1/2}_1)^{\dagger}\|\leq \sqrt{\frac{I+A~R^2}{A~R^2}}
\end{equation}
for our model problem.
Set
$
\fk:=\|(\mH^{1/2}_e\mH^{-1/2}_1)^{\dagger}\|
$
then
$$
\|(\mH^{1/2}_e\mH^{-1/2}_1)^*q_f\|=\sup_{\mathbf{v}\in\q(\fh_b)}
\frac{|(q_f,\mH^{1/2}_e\mathbf{v})|}{\|\mH^{1/2}_1\mathbf{v}\|}\geq \frac{1}{\fk}~\|P_{\q_\infty}q_f\|,
$$
since
$$
\je((\mH^{1/2}_e\mH^{-1/2}_1)^*)=\je(~\overline{\mH^{-1/2}_1\mH^{1/2}_e}~)
=\overline{\je(\mH^{1/2}_e)}=\overline{\q_\infty}^{~_{\|\cdot\|}}.
$$
For $Q^{-1}_eq_f \in L^2[0,l]$ we define
 $\mathbf{v}_0 = (\int_0^{({\cdot})} (Q^{-1}_eq_f) (s) ds, 0)$ (an element of $\q(h_\kappa)$).
For
general $\mathbf{v}$ we have
$$
\|\mH^{1/2}_1\mathbf{v}\|=\Big(E~I\int^l_0 \Big(\big[v_2'+\frac{v_1}{R}\big]'\Big)^2~ds+
E~A\int^l_0 \Big(v_1'-\frac{v_2}{R}\Big)^2~ds\Big)^{1/2}.
$$
Now, set $\mathbf{v}=\mathbf{v}_0$ and compute
$$
\|\mH^{1/2}_1\mathbf{v}_0\|=\frac{\sqrt{E~I+E~A~R^2}}{R}
\|q_f\|.
$$
This establishes
$$
\sup_{\mathbf{v}\in\q(\fh_b)}\frac{|(q_f,\mH^{1/2}_e\mathbf{v})|}{\|\mH^{1/2}_1\mathbf{v}\|}\geq
\frac{|(q_f,\mH^{1/2}_e\mathbf{v}_0)|}{\|\mH^{1/2}_1\mathbf{v}_0\|}\geq
\frac{|(Q^{-1}_eq_f,\mR_e\mathbf{v}_0)_{L^2}|}{\frac{\sqrt{E~I+E~A~R^2}}{R}\|q_f\|}
=\sqrt{\frac{A~R^2}{I+A~R^2}}\|q_f\|,
$$
which completes the proof of (\ref{trece:radiusocjena}).
\subsection{Quantitative (and qualitative) conclusions}\label{sec:qualitative}
The fact (\ref{trece:radiusocjena}) allows us to apply Theorem \ref{t:GT} to obtain precise
estimates for the behavior\footnote{We have tacitly dropped the exponent from $\eta_i(\varepsilon^{-2},\lambda_q^\infty)$
in order to simplify the notation.} of $\eta_i(\varepsilon,\lambda_q^\infty)$.
Since $\mH_1$ and not $\mH_b$ is
the positive definite operator, we will use the rod with the diameter $\varepsilon_0$ as a referent
configuration. We chose
\begin{equation}\label{eq:comput}
\varepsilon_0=\frac{\sqrt{3}}{6}\sqrt{\frac{I+A~R^2}{A~R^2}}\frac{\lambda_{\textsf{sec}}^\infty-\lambda_{\textsf{min}}^\infty}{
\lambda_{\textsf{sec}}^\infty+\lambda_{\textsf{min}}^\infty},
\end{equation}
where $\lambda_{\textsf{sec}}^\infty$ and $\lambda_{\textsf{min}}^\infty$
are the two lowermost eigenvalues of the Curved Rod model and
$\lambda_{\textsf{min}}^{\varepsilon}$ denotes the lowermost eigenvalue of the Arch Rod Model of the rod with
diameter $\varepsilon$.
Theorems \ref{t:Multi}, \ref{t:GT}, Remark
\ref{rem:below} and Corollary \ref{c:sharp}---together with the
observation that $\eta_i(\varepsilon,\lambda^\infty)<1$ for any $\varepsilon>0$---directly imply that
$$
\frac{\lambda_{\textsf{min}}^\infty-\lambda_{\textsf{min}}^{\varepsilon}}{\lambda_{\textsf{min}}^\infty}
\leq\varepsilon^2\frac{4 (I+A~R^2)}{ A~R^2}\frac{
\lambda_{\textsf{sec}}^\infty+\lambda_{\textsf{min}}^\infty}
{\lambda_{\textsf{sec}}^\infty-\lambda_{\textsf{min}}^\infty},\qquad 0<\varepsilon\leq\varepsilon_0.
$$
Furthermore, if we chose
$
\varepsilon_1=\frac{\sqrt{3}}{12}\sqrt{\frac{I+A~R^2}{A~R^2}}\frac{\lambda_{\textsf{sec}}^\infty-\lambda_{\textsf{min}}^\infty}{
\lambda_{\textsf{sec}}^\infty+\lambda_{\textsf{min}}^\infty},
$
then we obtain
\begin{equation}\label{eq:thin}
\varepsilon^2\frac{2 (I+A~R^2)\eta_i(\varepsilon_1,\lambda_{\textsf{min}}^\infty)}{ A~R^2}\leq\frac{\lambda_{\textsf{min}}^\infty-\lambda_{\textsf{min}}^{\varepsilon}}{\lambda_{\textsf{min}}^\infty}
\leq\varepsilon^2\frac{4 (I+A~R^2)}{ A~R^2}\frac{
\lambda_{\textsf{sec}}^\infty+\lambda_{\textsf{min}}^\infty}
{\lambda_{\textsf{sec}}^\infty-\lambda_{\textsf{min}}^\infty},
\end{equation}
$0<\varepsilon\leq\varepsilon_1$. If we are only interested in the upper estimate and we assume that there is $m\in\N$
such that $\lambda_m^\infty<\lambda_{m+1}^\infty$, then we have
\begin{equation}\label{eq:esti}
\frac{\lambda_i^\infty-\lambda_i^\varepsilon}{\lambda_i^\infty}\leq
\frac{3}{{\displaystyle \max_{i=1,...,m}\min_{k\ne
i}\frac{|\lambda_{k}^\infty-\lambda_{i}^\infty|}{\lambda_{k}^\infty+\lambda_{i}^\infty}}}
\frac{4 (I+A~R^2)}{ A~R^2}\varepsilon^2,\qquad i=1,\ldots,m.
\end{equation}
This estimate holds for all $\varepsilon\leq\varepsilon_2$, where $\varepsilon_2$ is defined as the first $\varepsilon$
for which the righthand side of (\ref{eq:esti}) is smaller then $1$.
Estimate (\ref{eq:esti}) can naturally be refined with the use of other $\eta_i(\varepsilon_2,\lambda_j^\infty)$ as
is given by the framework of Theorem \ref{tm:ess}. The optimality of the estimate is meant in the sense
of Theorem \ref{t:exactness}.

\section{Conclusion}
We have presented a constructive approach to spectral asymptotic estimates in the large coupling limit. Although
we have concentrated on a use of this results for theoretical considerations from
\cite{BruneauCarbou,Dancer,DemuthJeskeKirsch,SanchezPalencia90},
they are expected to be particularly useful in a
design of computational procedures for various singularly perturbed spectral problems.
This can be illustrated when comparing the numerical procedures for the Arch Model and the Curved Rod Model.
It has been shown that the Curved Rod Model is
is better behaved, with respect to the finite element approximations than the Arch Model,
see \cite{CurvedRodNumeric}. Furthermore, a qualitative conclusion of Section \ref{sec:qualitative} is
that when interested in the transversal vibrations only, Arch Model can be ignored (up to the corrections
of order $\varepsilon^2$).
For more on the lower dimensional approximations in the theory of
elasticity see \cite{CiarletV4y78,JurakTambaca,SanchezPalencia90,Tambaca1D,TamTeza}.

In a practical computational setting it is not
reasonable to assume that the spectral problem for $\mH_\infty$ will be exactly solvable. We would like
to emphasize that in the design of this theory we have not built the requirement of the explicit
solvability of $\mH_\infty$ into our results.
To be more precise, nowhere in the proofs of Theorems \ref{t:strong} and \ref{t:Multi}
or Corollaries \ref{c:working} and \ref{c:lower} is it necessary to have $\ra(P)=\fE_\infty$.
The only place where this assumption was necessary was to establish that (\ref{eq.sing_val_2})
and (\ref{eq:AppDef}) define the same approximation defects. Theorem \ref{tm:ess} and similar results from
\cite{Gru05_3, Gru_Ves_Sylv}---which are the workhorses of this theory---do not need this assumptions. Subsequently,
 the only limiting factor is the computability of $\eta_i(P)$ and the availability of information
on the distance of $\spec(\Xi_\kappa)$---from Theorems \ref{thm:second} and \ref{tm:ess}---
to the unwanted component of the spectrum.
With this we hope to have illustrated the advantages and limitations of our theory
\subsection*{Acknowledgement}
The author would like to thank V. Enss, Aachen for the support and for many helpful discussions during the
final preparation of this manuscript. The help of J. Tamba\v{c}a, Zagreb in the proof of the upper estimate
in Theorem \ref{t:GT} as well as in the proof of estimate (\ref{trece:radiusocjena}) is gratefully acknowledged.
The use of the technique of Lagrange multipliers in this context I have learned and adapted
from his papers. The author would also like to thank K. Veseli\'{c}, Hagen for
support and encouragement in early phases of the preparation of this manuscript
and for many useful discussions in this and other contexts.

\bibliographystyle{abbrv}
\def\cprime{$'$} \def\cprime{$'$} \def\cprime{$'$}

\end{document}